\documentclass[a4paper,10pt,onecolumn]{article}

\usepackage[tmargin=2cm,lmargin=2cm,rmargin=2cm,bmargin=2.5cm]{geometry}                		
\usepackage[noadjust]{cite}
\usepackage{amsmath,amssymb,amsfonts}
\usepackage{amsthm}
\usepackage{algorithmic}
\usepackage{graphicx}
\usepackage{textcomp}
\usepackage{mathrsfs}
\usepackage{xcolor}
\usepackage{enumerate}
\usepackage{booktabs,tabularx}

\newtheorem{theorem}{Theorem}
\newtheorem{lemma}[theorem]{Lemma}
\newtheorem{corollary}[theorem]{Corollary}
\newtheorem{proposition}[theorem]{Proposition}
\newtheorem{remark}[theorem]{Remark}
\newtheorem{assumption}{Assumption}

\def\bpf{\begin{proof}}
\def\epf{\end{proof}}

\numberwithin{equation}{section}

\allowdisplaybreaks[4]

\DeclareMathOperator*{\argmin}{argmin}

\DeclareMathOperator{\interior}{int}
\newcommand{\defeq}{:=}

\DeclareMathOperator{\Co}{Co}

\def\btheta{\bar{\theta}}

\def\RR{\mathbb{R}}
\def\NN{\mathbb{N}}
\def\R{\mathcal{R}}

\def\B{\mathcal{B}}
\def\N{\mathcal{N}}
\def\X{\mathcal{X}}
\def\S{\mathcal{S}}

\def\W{\mathcal{W}}
\def\M{\mathcal{M}}
\def\probw{p_w}
\def\hw{\hat{w}}

\def\Id{{\mathcal { {I}}}}
\def\ones{{\bf 1}}

\def\Wtight{\Omega}

\newcommand{\blue}[1]{\textcolor{black}{#1}}
\newcommand{\red}[1]{\textcolor{black}{#1}}

\begin{document}

\title{Robust Adaptive Model Predictive Control:
Performance and Parameter Estimation}

\author{Xiaonan Lu$\mbox{}^\ast$, Mark Cannon\thanks{Department of Engineering Science, University of Oxford, UK}~~and Denis Koksal-Rivet\thanks{Department of Mathematics, University of Chicago, Chicago, USA}}

\maketitle






 
\begin{abstract}
For systems with uncertain linear models, bounded additive disturbances and state and control constraints, a robust model predictive control algorithm incorporating online model adaptation is proposed.
Sets of model parameters are identified online and employed in a robust tube MPC strategy with a nominal cost. The algorithm is shown to be recursively feasible and input-to-state stable.
%
%
%
Computational tractability is ensured by using polytopic sets of fixed complexity to bound parameter sets and predicted states. 
Convex conditions for persistence of excitation are derived and are related to probabilistic rates of convergence and asymptotic bounds on parameter set estimates. 
We discuss how to balance conflicting requirements on control signals for achieving good tracking performance and parameter set estimate accuracy.
Conditions for convergence of the estimated parameter set are discussed for the case of fixed complexity parameter set estimates, inexact disturbance bounds and noisy measurements.

\noindent\textbf{keywords:}
Control of constrained systems, Adaptive control, Parameter set estimation, Receding horizon control, Stochastic convergence
\end{abstract}



\section{Introduction}
Model Predictive Control (MPC) repeatedly solves a finite-horizon optimal control problem subject to input and state constraints. At each sampling instant a model of the plant is used to optimize predicted behaviour and the first element of the optimal predicted control sequence is applied to the plant \cite{Mayne2000}. 
Any mismatch between model and plant causes degradation of controller performance \cite{Badwe2009}. As a result, the amount of model uncertainty strongly affects the bounds of the achievable performance of a robust MPC algorithm \cite{Kouvaritakis2015}.

To avoid the disruption caused by intrusive plant tests~\cite{Badwe2009}, adaptive Model Predictive Control attempts to improve model accuracy online while satisfying operating constraints and providing stability guarantees. Although the literature on adaptive control has long acknowledged the need for persistently exciting inputs for system identification~\cite{Narendra1987}, few papers have explored how to incorporate Persistency of Excitation (PE) conditions with feasibility guarantees 
within adaptive MPC~\cite{Mayne2014}. In addition, adaptive MPC algorithms must balance conflicting requirements for system identification accuracy and computational complexity \cite{Mayne2014,Qin2003}.

Various methods for estimating system parameters and meeting operating constraints are described in the adaptive MPC literature.
Depending on the assumptions on model parameters, parameter identification methods such as {recursive least squares \cite{Heirung2017}}, comparison sets \cite{Aswani2013}, set membership identification \cite{Tanaskovic2014,Lorenzen2018} and neural networks \cite{Akpan2011,Reese2016} have been proposed. 
Heirung et al.~\cite{Heirung2013} propose an algorithm where the unknown parameters are estimated using recursive least squares (RLS) and system outputs are predicted using the resulting parameter estimates. The use of RLS introduces nonlinear equality constraints into the optimisation. 
On the other hand, the comparison model approach described in Aswani et al.~\cite{Aswani2013} addresses the trade-off between probing for information and output regulation by decoupling these two tasks; a nominal model is used to impose operating constraints whereas performance is  evaluated via a model learned online using statistical identification tools.
However the use of a nominal model implies that the comparison model approach
cannot guarantee robust constraint satisfaction. 

Tanaskovic et al.~\cite{Tanaskovic2014} consider a linear Finite Impulse Response (FIR) model with measurement noise and constraints. This approach updates a model parameter set using online set membership identification; constraints are enforced for the entire parameter set and performance is optimized for a nominal prediction model. The paper proves recursive feasibility but does not show convergence of the identified parameter set to the true parameters. To avoid the restriction to FIR models, Lorenzen et al.~\cite{Lorenzen2018} consider a linear state space model with additive disturbance. An online-identified set of possible model parameters is used to robustly stabilize the system.
However the approach suffers from a lack of flexibility in its robust MPC formulation, which is based on homothetic tubes~\cite{Rakovic2012}, allowing only the centers and scalings of tube cross-sections to be optimized online, and it does not provide convex and recursively feasible conditions to ensure persistently exciting control inputs. 

In this paper we also consider linear systems with parameter uncertainty, additive disturbances and constraints on system states and control inputs. Compared with \cite{Lorenzen2018}, the proposed algorithm reduces the conservativeness in approximating predicted state tubes by adopting more flexible cross-section representations. Building on \cite{Lu2019}, we take advantage of fixed complexity polytopic tube representations and use hyperplane and vertex representations interchangeably to further simplify computation. 
We use, similarly to \cite{Heirung2013}, a nominal performance objective
 {, but we impose constraints}
robustly on all possible models within the identified model set. We prove that the closed loop system is input-to-state stable (ISS). In comparison with the min-max approach of \cite{Lu2019}, the resulting performance bound takes the form of an asymptotic bound on the 2-norm of the sequence of closed loop states in terms of the 2-norms of the additive disturbance and parameter estimate error sequences. 
\blue{In addition, we convexify the persistence of excitation (PE) condition around a reference trajectory and include a penalty term in the cost function to promote convergence of the parameter set. The convexification method is somewhat analogous to that proposed in \cite{Ferizbegovic2020,Iannelli2019}, where the uncertainty information of the parameter set is approximated using a nominal gain. Here however the convexification is obtained by direct linearization of PE constraints on predicted trajectories. The cost function modification proposed here allows the relative importance of the two objectives, namely controller performance and convergence of model parameters, to be specified.}

Bai et al.~\cite{Bai1998} consider a particular set membership identification algorithm and show that the parameter set estimate converges with probability~1 to the actual parameter vector (assumed constant) if: (a) a tight bound on disturbances is known; (b) the input sequence is persistent exciting and (c) the minimal parameter set estimate is employed.
However the minimal set estimate can be arbitrarily complex, and to provide computational tractability various non-minimal parameter set approximations have been proposed, such as $n$-dimensional balls \cite{Adetola2009} and bounded complexity polytopes \cite{Tanaskovic2014}. The current paper allows the use of parameter set estimates with fixed complexity and proves that, despite their approximate nature, such parameter sets converge with probability~1 to the true parameter values. We also derive lower bounds on convergence rates for the case of inexact knowledge of the disturbance bounding set and for the case that model states are estimated in the presence of measurement noise. 

This paper has five main parts. Section \ref{sec:problem_definition} defines the problem and basic assumptions. Section \ref{sec:AMPC} gives details of the parameter estimation, robust constraint satisfaction, nominal cost function, convexified PE conditions and the MPC algorithm. Section \ref{sec:feasibility_stability} proves recursive feasibility and input-to-state stability of the proposed algorithm. Section \ref{sec:convergence_parameter} proves the convergence of the parameter set in various conditions and Section \ref{sec:example} illustrates the approach with numerical examples.


%


\textit{Notation:} 
$\mathbb{N}$ and $\mathbb{R}$ denote the sets of integers and reals, and 
$\mathbb{N}_{\geq 0} = \{n\in\mathbb{N} : n \geq 0\}$,
$\mathbb{N}_{[p,q]} = \{n\in\mathbb{N} : {p \leq n \leq q}\}$. The $i$th row of a matrix $A$ and $i$th element of a vector $a$ are denoted $[A]_i$ and $[a]_i$. Vectors and matrices of $1$s are denoted $\ones$, and $\Id$ is the identity matrix. 
For a vector $a$, $\| a \|$ is the Euclidean norm and $\| a\|^2_P = a^\top P a$; the largest element of $a$ is $\max a$ and $[a]_{\geq 0}=\max \{ 0,a \}$. The absolute value of a scalar $s$ is $\lvert s \rvert$ and the floor value is $\lfloor s \rfloor$. 
$\lvert \mathcal{S} \rvert$ is the number of elements in a set $\mathcal{S}$. $\mathcal{A}\oplus \mathcal{B}$ is Minkowski addition for sets $\mathcal{A}$ and $\mathcal{B}$, and $\mathcal{A}\oplus \mathcal{B} = \{a+b: a\in \mathcal{A},b\in \mathcal{B}\}$.
The matrix inequality
$A \succeq 0$ (or $A \succ 0$) indicates that $A$ is positive semidefinite (positive definite) matrix. 
The $k$ steps ahead predicted value of a variable $x$ is denoted $x_k$, and the more complete notation $x_{k|t}$ indicates the $k$ steps ahead prediction at time $t$. 
%
A continuous function $\sigma: \RR_{\geq 0} \rightarrow \RR_{\geq 0}$ is a $\mathcal{K}$-function if it is strictly increasing with $\sigma(0)=0$, 
and is a $\mathcal{K}_\infty$-function if in addition 
$\sigma(s)\rightarrow \infty$ as $s\rightarrow \infty$. 
A continuous function $\phi:{\mathbb R}_{\geq 0} \times{\mathbb R}_{\geq 0} \rightarrow {\mathbb R}_{\geq 0}$ is a  ${\mathcal K} {\mathcal L}$-function if, for all $t\geq 0$, $\phi(\cdot,t)$ is a ${\mathcal K}$-function, and, for all $s\geq 0$, $\phi(s,\cdot)$ is decreasing with $\phi(s,t)\rightarrow 0$ as $t\rightarrow \infty$.
For functions $\sigma_a$ and $\sigma_b$ we denote $\sigma_a \circ \sigma_b (\cdot) = \sigma_a\bigl(\sigma_b(\cdot)\bigr)$, and $\sigma_a^{k+1}(\cdot) = \sigma_a \circ \sigma_a^{k}(\cdot)$ with \mbox{$\sigma_a^1(\cdot) = \sigma_a(\cdot)$}. 


\section{Problem formulation and preliminaries} \label{sec:problem_definition}
This paper considers a linear system with linear state and input constraints and unknown additive disturbance:
\begin{equation} \label{eq:update_equation}
x_{t+1} = A(\theta^\ast)x_t + B(\theta^\ast) u_t +w_t ,
\end{equation}
where $x_t \in \mathbb{R}^{n_x}$ is the system state, $u_t \in \mathbb{R}^{n_u}$ is the control input, $w_t \in \mathbb{R}^{n_x}$ is an unknown disturbance input, and $t$ is the discrete time index. The system matrices $A(\theta^\ast)$ and $B(\theta^\ast)$ depend on an unknown but constant parameter $\theta^\ast\in\mathbb{R}^p$.
The disturbance sequence $\{w_0,w_1,\ldots\}$ is stochastic and $(w_i,w_j)$ is independent for all $i\neq j$.
States and control inputs are subject to linear constraints, defined for $F\in\mathbb{R}^{n_c\times n_x}$, $G\in\mathbb{R}^{n_c\times n_u}$ by
\begin{equation}\label{eq:constraints}
Fx_t +Gu_t \leq \ones \quad \forall  t \in \mathbb{N}_{\geq 0} .
\end{equation}

\begin{assumption}[Additive disturbance] \label{ass:compact_disturbance}
The disturbance $w_t$ lies in a convex and compact polytope $\W$, where
\begin{equation} \label{eq:w _set}
\W = \{ w : \Pi_w w \leq \pi_w \} 
\end{equation}
with $\Pi_w \in \mathbb{R}^{n_w \times n_x} $, $\pi_w \in \mathbb{R}^{n_w}$ and $\pi_w > 0$.
\end{assumption}

\begin{assumption}[Parameter uncertainty]\label{ass:param_uncertainty}%
The system matrices $A$ and $B$ are affine functions of the parameter vector $\theta \in \mathbb{R}^{p}$:
\begin{equation} \label{eq:system_model}
(A(\theta),B(\theta)) = (A_0,B_0) +\sum_{i = 1}^{p} (A_i,B_i) [\theta]_i
\end{equation}
for known matrices $A_j$, $B_j$, $j \in\mathbb{N}_{[1,p]}$,
and $\theta^\ast$ lies in a known, bounded, convex polytope $\Theta_0$ given by
\[
\Theta_0 = \{ \theta : M_\Theta \theta \leq \mu_{0} \} .
\]
\end{assumption}



\begin{assumption}[State and control constraints]%
\label{ass:compact_feasible_region}
The set 
\[
\mathcal{Z} = \{(x,u) \in \RR^{n_x} \times \RR^{n_u}: Fx + Gu \leq \ones \}
\] 
is compact and contains the origin in its interior.
\end{assumption}

To obtain finite numbers of decision variables and constraints in the MPC optimization problem, the predicted control sequence at time $t$ is assumed to be expressed in terms of optimization variables $v_{0|t},\ldots,v_{N-1|t}$ as
%
\begin{equation} \label{eq:input_law}
u_{k|t} = \begin{cases}Kx_{k|t}+v_{k|t} & \forall k \in \mathbb{N}_{[0,N-1]} 
\\ 
K x_{k|t} & \forall k  \geq N
\end{cases} 
\end{equation}
where $N$ is the prediction horizon. 
The gain $K$ is designed offline and is assumed to robustly stabilize the uncertain system $x_{t+1} = (A(\theta) + B(\theta)K) x_t$, $\forall \theta\in\Theta_0$ in the absence of constraints.  {This assumption can be stated as follows.}

\begin{assumption}[Feedback gain and contractive set] \label{ass:lambda_contractive}
There exists a polytopic set ${\X}=\{ x : T x \leq \ones\}$ and feedback gain $K$ such that ${\X}$ is $\lambda$-contractive for some $\lambda \in [0,1)$, i.e.
\begin{equation}\label{eq:lambda}
T \bigl(A(\theta) + B(\theta)K\bigr)x \leq \lambda \ones 
\end{equation}
for all $x\in\{x: Tx \leq \ones \}$ and $\theta\in\Theta_0$. The representation ${\X}=\{ x : T x \leq \ones\}$ is assumed to be minimal in the sense that it contains no redundant inequalities. 
\end{assumption}


\section{Adaptive Robust MPC} \label{sec:AMPC}

In this section a parameter estimation scheme based on~\cite{Chisci1998,Veres1999} is introduced. We then discuss the construction of tubes to bound predicted model states and associated constraints.

\subsection{Set-based parameter estimation}\label{sec:fixed_comp_set_update}

At time $t$ we use observations of the system state $x_t$ to determine a set $\Delta_t$ of unfalsified model parameters. The set $\Delta_t$ is then combined with the parameter set estimate $\Theta_{t-1}$ to construct a new parameter set estimate $\Theta_t$.

\textit{Unfalsified parameter set:}~
Define $D_t $ and $d_t $ as the matrix and vector
\begin{align}
D_t & =D(x_t,u_t)  = \begin{bmatrix} A_1 x_t+B_1 u_t &  \cdots & A_p x_t + B_p u_t\end{bmatrix} 
\label{eq:Dk}
\\
d_t & =d(x_t,u_t) = A_0 x_{t}+B_0 u_{t} . \label{eq:dk}
\end{align}
Then, given $x_t$, $x_{t-1}$, $u_{t-1}$ and the disturbance set $\W$ in (\ref{eq:w _set}), the unfalsified parameter set at time $t$ is given by
\begin{equation} \label{eq:delta_t_set}
\Delta_t = \{ \theta : x_t -A(\theta)x_{t-1} -B(\theta)u_{t-1} \in \W\}
=\{\theta : P_t \theta \leq q_t\} 
\end{equation}
with $P_t = {-\Pi_w D_{t-1} }$ and $q_t = \pi_w +\Pi_w (d_{t-1}-x_t)$.

\textit{Parameter set update:}~
Let $M_\Theta\in\RR^{r\times p}$ be an \textit{a priori} chosen matrix. 
The estimated parameter set $\Theta_t$ is defined by
\begin{equation} \label{eq:defn_Theta_set}
\Theta_t  = \Theta(\mu_t) = \{ \theta : M_{\Theta} \theta \leq \mu_{t} \}
\end{equation}
where  {$\mu_t \in \RR^r$} is updated online at times $t\in\mathbb{N}_{\geq 0}$. 
The complexity of $\Theta_t$ is controlled by fixing $M_\Theta$, which fixes the directions of the half-spaces defining the parameter set. 
We assume that $M_\Theta$ is chosen so that $\Theta_t$ is compact for all $\mu_t$ such that $\Theta_t\neq \emptyset$.
Using a block recursive polytopic update method \cite{Chisci1998}, $\Theta_{t}$ is defined as the smallest set (\ref{eq:defn_Theta_set}) containing the intersection of $\Theta_{t-1}$ and unfalsified sets $\Delta_j$ over a window of length $N_u$:
\begin{equation} \label{eq:fixed_set_update}
\mu_{t} = \min_{\mu\in\RR^r} \text{vol} \bigl(\Theta(\mu)\bigr)
\ \ \text{subject to}\ \ 
\Theta(\mu) \supseteq \bigcap_{j=t-N_u+1}^t  \Delta_j \cap \Theta_{t-1}
\end{equation}
(where $\Delta_j = \RR$ for all $j\leq0$). We refer to $N_u$ as the PE window.
Note that $N_u$ is independent of the MPC prediction horizon $N$. 
%
Using linear conditions for polyhedral set inclusion~\cite{Blanchini}
$\mu_{t}$ in (\ref{eq:fixed_set_update}) can be obtained by solving a linear program for each $i \in \mathbb{N}_{[1,r]}$:
\[
[ \mu_{t} ]_i = \min_{\mu, \, H_i} \mu 
\ \ \text{subject to} \ \ 
H_i  \begin{bmatrix} M_{\Theta} \\ P_{t-N_u+1}\\ \vdots \\ P_t \end{bmatrix} =  [M_{\Theta}]_i,\ H_i \begin{bmatrix} \mu_{t-1} \\ q_{t-N_u+1} \\ \vdots \\ q_t \end{bmatrix} \leq \mu, \  H_i \geq 0.
\]

%

\begin{lemma}\label{lem:monotonic_param_set}
If $\theta^\ast\in\Theta_0$ 
and $\Theta_{t}$ is defined by (\ref{eq:defn_Theta_set}), (\ref{eq:fixed_set_update}), then 
$\theta^\ast \in \Theta_t $ and
$\Theta_{t} \supseteq \Theta_{t+1} \supseteq (\Theta_{t} \cap \Delta_{t+1})$ for all $t\in\mathbb{N}_{\geq 0}$.
\end{lemma}

\subsection{Polytopic tubes for robust constraint satisfaction}
This section considers predicted state and control trajectories. To simplify notation, we omit the subscript $t$ indicating the time at which state and control predictions are made, 
\blue{whenever $t$ indicates current time}; thus 
the $k$ steps ahead predictions $x_{k|t}$, $v_{k|t}$ are denoted $x_k$, $v_k$.
To ensure that the predicted state and control sequences satisfy the operating constraints~(\ref{eq:constraints}) robustly for the given uncertainty bounds, we construct a tube (a sequence of sets) $\X_0,\X_1,\ldots \subset \mathbb{R}^{n_x}$ satisfying, 
for all $x\in \X_k$, $w\in\W$, $\theta\in\Theta_t$,
\begin{equation} 
\label{eq:recurrence_constraint}
\bigl(A(\theta) + B(\theta)K\bigr)x + B(\theta) v_k + w \in \X_{k+1} 
\ \ \forall k\in\mathbb{N}_{\geq 0} .
\end{equation}

\textit{Hyperplane form:}~
For given $T\in\mathbb{R}^{n_\alpha\times n_x}$ satisfying Assumption~\ref{ass:lambda_contractive} and $\alpha_k\in\mathbb{R}^{n_\alpha}$, let $\X_{k}\subset\RR^{n_x}$ denote the $k$ steps ahead cross-section of the predicted state tube:
\begin{equation} \label{eq:T_set}
\X_{k} = \{ x : T x \leq \alpha_k\},
\end{equation}

 {
The MPC algorithm described in Section~\ref{subsec:proposed_algorithm} optimizes the shape of the predicted state tube online by allowing $\alpha_k$ to be an optimization variable.
If, for a given $\alpha_k$, the constraint $[T]_i x \leq [\alpha_k]_i$ is redundant for some $i\in\NN_{[1,n_\alpha]}$ in the hyperplane description~(\ref{eq:T_set}) (i.e.~if the set $\X_k$ is unchanged by removing this constraint), we define (without loss of generality) $[\alpha_k]_i = \max_{x\in \X_k} [T]_i x$. 
Thus, for each $i\in\NN_{[1,n_\alpha]}$, $[\alpha_k]_i = [T]_i x$ necessarily holds for some $x\in\X_k$.} 
Then (\ref{eq:recurrence_constraint}) is equivalent to, for all $x\in \X_k$ and $\theta\in\Theta_t$,
\[
T\bigl(  A(\theta) +  B(\theta) K \bigr) x + T B(\theta) v_k + \bar{w} \leq \alpha_{k+1} , \ \ \forall k\in\mathbb{N}_{\geq 0} 
\]
where $\bar{w}$ is the vector with $i$th element $ [\bar{w}]_i = \max_{w\in\W} [T w]_i$ for all $i \in \mathbb{N}_{[1, n_{\alpha}]}$. 
Substituting $D(x,u)$ and $d(x,u)$ from (\ref{eq:Dk}), (\ref{eq:dk}), this implies linear conditions on $\theta$, for all $x\in \X_k$, $\theta\in\Theta_t$:
\[
T \big( D( x,Kx+v_k )\theta + d( x,Kx+v_k ) \big)
+ \bar{w} 
\leq \alpha_{k+1}, 
\ \ \forall k\in\mathbb{N}_{\geq 0} 
\]
and, for a given initial state $x$, the constraint $x\in\X_0$ requires  
%
\begin{equation} \label{eq:initial_con}
Tx \leq \alpha_0 .
\end{equation}

\textit{Vertex form:}~
$\X_{k}$ has an equivalent representation in terms of its vertices, which we denote as $x_k^{(j)}$, $j\in\mathbb{N}_{[1,m]}$:
\begin{equation}\label{eq:vertex_form} 
\X_{k} = \Co\{ x_{k}^{(1)}, \dots,  x_{k}^{(m)} \} .
\end{equation}

 {Note that $m$, the number of vertices of $\X_k$, is fixed, and for given $\alpha_k$ in (\ref{eq:T_set}), we may have  $x_k^{(i)}  = x_k^{(j)}$ for some $i \neq j \in\mathbb{N}_{[1,m]}$ in (\ref{eq:vertex_form}) (i.e.~the vertex description may contain repeated vertices). The presence of repeated vertices does not affect the formulation.}
For each $j \in \mathbb{N}_{[1, m]}$, define an index set $\R_j$  (with $\lvert \R_j \rvert = n_x$),
such that the $i$th row of $T$ and $i$th element of $\alpha_{k}$ satisfy
\[
[T]_i x_{k}^{(j)} = [\alpha_{k}]_i \ \forall i\in\R_j .
\]
Since $T$ is constant, 
the index set $\R_j$ associated with active inequalities at the vertex $x^{(j)}$ is
independent of $\alpha_{k}$ and can be computed offline. Therefore, for each $j \in \mathbb{N}_{[1, m]}$, we have
\begin{equation} \label{eq:X_defined_by_U}
x_{k}^{(j)} = U_j \alpha_{k},
\end{equation}
where the matrix $U_j\in\RR^{n_x\times n_{\alpha}}$ can be computed offline given knowledge of $\R_j$ using the property that
\begin{equation} \label{eq:define_U}
[T]_i U_j  = [\Id]_i \ \forall i \in \R_j , \ j\in\mathbb{N}_{[1,m]}.
\end{equation}
Using the vertex representation (\ref{eq:X_defined_by_U}), the condition that (\ref{eq:constraints}) is satisfied for all $x\in\X_k$ is equivalent to, for all $j \in \mathbb{N}_{[1, m]}$,
\begin{equation} \label{eq:input_state_constraint}
(F+GK) U_j \alpha_{k} + G v_{k} \leq \ones .
\end{equation}
 {Substituting $D(x,u)$ and $d(x,u)$ from (\ref{eq:Dk}), (\ref{eq:dk}), }
condition~(\ref{eq:recurrence_constraint}) can be expressed equivalently as
\[
T \bigl( D( U_{j}\alpha_{k}, K U_{j}\alpha_{k} + v_{k} )\theta + d( U_{j}\alpha_{k}, K U_{j}\alpha_{k} + v_{k} ) \bigr) + \bar{w}  \leq \alpha_{k+1}
\]
for all $\theta\in\Theta_t$, $j \in \mathbb{N}_{[1, m]}$ and $k\in\NN_{\geq 0}$.
This is equivalent~\cite[Prop.~3.31]{Blanchini} to the requirement that there exist matrices $\Lambda_{k,j}$ satisfying, for each prediction time step $k \in \mathbb{N}_{\geq0}$ and each 
vertex $j \in \mathbb{N}_{[1, m]}$, the conditions
\begin{subequations}\label{eq:update_constraint}
\begin{align}
&\Lambda_{k,j} M_{\Theta} = T D(U_j\alpha_{k},K U_j\alpha_{k} + v_{k} )\\
&\Lambda_{k,j} \mu_t \leq \alpha_{k+1} - T d( U_j\alpha_{k}, K U_j\alpha_{k} + v_{k} ) - \bar{w} \\
&\Lambda_{k,j} \geq 0 .
\end{align}
\end{subequations}
Given the dual mode predicted control law~(\ref{eq:input_law}), we introduce the terminal conditions that $\bigl(A(\theta)+B(\theta)K\bigr) x + w\in \X_N$ and $(F+GK)x \leq \ones$ for all $x\in\X_N$, $w\in\W$ and $\theta\in\Theta$. Then (\ref{eq:constraints}) and (\ref{eq:recurrence_constraint}) are satisfied if 
(\ref{eq:input_state_constraint}), (\ref{eq:update_constraint}) hold for all $j \in \mathbb{N}_{[1, m]}$ and $k \in \mathbb{N}_{[0,N-1]}$, and
there exist matrices $\Lambda_{N,j}$ satisfying the conditions, for all $j \in \mathbb{N}_{[1, m]}$
\begin{subequations}\label{eq:final_tube_condition}
\begin{align}
&(F+GK) U_j \alpha_{N}  \leq \ones\\
&\Lambda_{N,j} M_{\Theta} = T D(U_j\alpha_{N},K U_j\alpha_{N} )\\
&\Lambda_{N,j} \mu_t \leq \alpha_{N} - T d( U_j\alpha_{N}, K U_j\alpha_{N}) - \bar{w} \\
&\Lambda_{N,j} \geq 0 .
\end{align}
\end{subequations}

\subsection{Objective function}
Consider the nominal cost defined for $Q,R\succ 0$ by
\begin{equation} \label{eq:nominal_cost_function}
J(x,\mathbf{v},\bar{\theta}_t) =  \sum_{k = 0}^{\infty} (\|\bar{x}_{k} \|_Q^2 +\| \bar{u}_{k} \|_R^2), 
\end{equation}
where $\bar{x}_k$ and $\bar{u}_k$ are elements of predicted state and control sequences generated by a nominal parameter vector $\bar{\theta}_t$:
\begin{subequations} \label{eq:nominal_update}
	\begin{align}
	& \bar{x}_{0} = x \\
	&\bar{x}_{k+1} = \bigl(A(\bar{\theta}_t) +B(\bar{\theta}_t)K\bigr) \bar{x}_{k} + B(\bar{\theta}_t) v_{k} \\
	&\bar{u}_{k} = \blue{\begin{cases} K \bar{x}_{k}+{v}_{k} & k< N\\ K \bar{x}_k & k \geq N \end{cases}}
	\end{align}
\end{subequations}
 {for $k\in \NN_{>0}$} and where ${\bf v} = \{v_{0},\ldots,v_{N-1}\}$.
Define $P(\theta)$ as the solution of the Lyapunov matrix equation 
\begin{equation}\label{eq:lyap_mat}
P(\theta) - \Phi(\theta)^\top P(\theta) \Phi(\theta) = Q + K^\top R K
\end{equation}
where $\Phi(\theta) = A(\theta) + B(\theta)K$. Note that $P(\theta)\succ 0$ is well-defined for all  $\theta\in\Theta_0$ due to Assumption~\ref{ass:lambda_contractive}. 
Then (\ref{eq:nominal_cost_function}) is equivalent to
\begin{equation}\label{eq:nominal_dual_mode_cost}
J(x,\mathbf{v},\bar{\theta}_t)  =  \sum_{k = 0}^{N-1} (\|\bar{x}_{k} \|_Q^2 +\| \bar{u}_{k} \|_R^2) +\| \bar{x}_{N} \|_{P(\bar{\theta}_t)}^2 .
\end{equation}
We assume knowledge of an initial nominal parameter vector $\bar{\theta}_0\in\Theta_0$, which could be estimated using physical modeling or offline system identification, alternatively $\bar{\theta}_0$ could be defined as the Chebyshev centre of $\Theta_0$.
%
For $t>0$, we assume that $\bar{\theta}_t$ is updated by projecting $\bar{\theta}_{t-1}$ onto the parameter set estimate $\Theta_t$, i.e. 
\begin{equation} \label{eq:nominal_theta_update}
\bar{\theta}_{t} = \argmin_{\theta\in\Theta_t} \|\bar{\theta}_{t-1} - \theta\| .
\end{equation}
\begin{remark}
For the \blue{input-to-state stability} analysis in Section~\ref{sec:feasibility_stability} it is essential that $\bar{\theta}_t \in \Theta_t$. However, subject to this constraint, alternative update laws for $\bar{\theta}_t$ are possible; for example a  {Least Mean Squares (LMS) estimate} projected onto $\Theta_t$~\cite{Lorenzen2018}. 
\end{remark}

\subsection{Augmented objective function and persistent excitation}
\label{sec:pe_cost}
The regressor $D_t$ in (\ref{eq:Dk}) is persistently exciting (PE) if 
\begin{equation}\label{eq:twosided_pe_condition}
\beta_1 \Id \preceq  {\sum_{t=t_0}^{t_0+N_u-1} D_t^\top D_t }\preceq \beta_2 \Id 
\end{equation}
for some PE window $N_u \in\mathbb{N}_{>0}$, some $\beta_2\geq\beta_1>0$, and all times $t_0$~\cite{Narendra1987}. Although the upper bound in (\ref{eq:twosided_pe_condition}) implies convex constraints on $x_t$ and $u_t$, the lower bound is nonconvex. 
 {The bounds on convergence of the parameter set $\Theta_t$ derived in Section~\ref{sec:convergence_parameter} suggest faster convergence as $\beta_1$ in the PE condition~(\ref{eq:twosided_pe_condition}) increases.}

 {
Previously proposed MPC strategies that incorporate persistency of excitation constraints consider the PE condition to be defined on an interval such as $\{t-N_u+1,\ldots,t\}$, where $t$ is current time, which means that the PE constraint depends on  only the first element of the predicted control sequence. Marafioti et al.~\cite{Marafioti2014} simplify the PE condition by expressing it as a nonconvex quadratic inequality in $u_t$. Likewise, Lorenzen et al.~\cite{Lorenzen2018} show that the PE condition is equivalent to a nonconvex constraint on the current control input.
Lu and Cannon \cite{Lu2019} linearize the PE condition about a reference trajectory and thus obtain a sufficient condition for persistent excitation.}

 {In this paper, on the hand, we define the PE condition} over predicted trajectories (from $k = 0$ to $k=N_u -1$ steps ahead), and we therefore
require, at time $t$ and for some $\beta_1 > 0$,
\begin{equation}\label{eq:PE_condition}
 {\sum_{k = 0}^{N_u -1} D_{k|t}^\top D_{k|t} \succeq  \beta_1 \Id}.
\end{equation}
 {The inclusion of predicted future states and control inputs in this PE condition allows for greater flexibility in meeting the constraint.}
To avoid nonconvex constraints, we derive a convex relaxation that provides a sufficient condition for (\ref{eq:PE_condition}).

Assume reference state and control predicted sequences,
$\hat{\mathbf{x}} = \{\hat{x}_0,\ldots \hat{x}_{N}\}$ and
$\hat{\mathbf{u}} = \{\hat{u}_0,\ldots \hat{u}_{N-1}\}$, 
approximating the optimal predicted state and control sequences, are available. To derive sufficient conditions for (\ref{eq:PE_condition}), we consider the difference between the reference and optimized sequences, denoted
$\tilde{u}_k$ and $\tilde{x}_k$ (i.e.  $\tilde{u}_k  = u_k - \hat{u}_k  { = K x_k -\hat{u}_k}$ and $\tilde{x}_k  = x_k - \hat{x}_k$). Since the prediction tube implies $x_k \in \X_k$, we therefore have $\tilde{x}_k \in \tilde{\X}_k$ where $\tilde{\X}_k =\X_k \oplus -\hat{x}_k $. Denote the vertices of $\tilde{\X}_k$ as $\tilde{x}_k^{(j)}$ and for $ j \in \NN_{[1, m]}$, we have
\begin{equation*}
\tilde{x}_k^{(j)}= {x}^{(j)}_{k}- \hat{x}_{k} = U_j \alpha_k - \hat{x}_k.
\end{equation*}
%
Moreover $D_k$ depends linearly on ${u}_k$ and ${x}_k$, and hence for $ j \in \NN_{[1, m]}$,
\begin{multline*}
 {D(x_k^{(j)}, u_k^{(j)})^\top D(x_k^{(j)}, u_k^{(j)}) }
= D(\hat{x}_k,\hat{u}_k)^\top D(\hat{x}_k,\hat{u}_k)  +D(\tilde{x}_k^{(j)}, \tilde{u}_k^{(j)}) ^\top  D(\hat{x}_k,\hat{u}_k)
\\
+ D(\hat{x}_k,\hat{u}_k)^\top D(\tilde{x}_k^{(j)}, \tilde{u}_k^{(j)})  
+D(\tilde{x}_k^{(j)}, \tilde{u}_k^{(j)})^\top D(\tilde{x}_k^{(j)}, \tilde{u}_k^{(j)}) .
\end{multline*}
Here $D(\tilde{x}_k^{(j)}, \tilde{u}_k^{(j)})^\top D(\tilde{x}_k^{(j)}, \tilde{u}_k^{(j)}) $ is a positive semidefinite matrix, and by omitting this term we obtain sufficient conditions for (\ref{eq:PE_condition}) as a set of LMIs in $\alpha_k$ and $v_k$. 
{The following convex conditions are thus sufficient to ensure (\ref{eq:PE_condition}) whenever $\beta \geq \beta_1$}
 {
\begin{subequations}\label{eq:PE_constraint}
\begin{align}
D(\hat{x}_k,\hat{u}_k)^\top D(\hat{x}_k,\hat{u}_k)  
+&D(U_j \alpha_k - \hat{x}_k, KU_j \alpha_k + v_k - \hat{u}_k) ^\top  D(\hat{x}_k,\hat{u}_k) \nonumber
\\
& +   D(\hat{x}_k,\hat{u}_k)^\top D(U_j \alpha_k - \hat{x}_k, KU_j\alpha_k + v_k - \hat{u}_k)
\succeq  M_k,  
\ \ \forall{j \in \NN_{[1,m]}}, \ \forall{k \in \NN_{[0,N_u-1]}}
\\
&\qquad \qquad \qquad \qquad\qquad \qquad \qquad \quad  
\sum_{k = 0}^{N_u-1} M_k  \succeq  \beta \Id
\end{align}
\end{subequations}
where $M_k \in \RR^{p\times p}\ \forall{k \in \NN_{[0,N_u-1]}}$ are intermediate variables. 
}

 {
Another innovation of this paper is the inclusion of PE coefficient in the cost function. 
Previous approaches~\cite{Lorenzen2018,Lu2019} face the difficulty of choosing a suitable $\beta$ value for the PE constraint in the implementation. A larger value of $\beta$ is generally desirable, but a large $\beta$ might make the optimisation problem with the PE condition infeasible. }
 {In this paper we incentivize a large value of $\beta$ by modifying the MPC objective function as follows}
\begin{equation} \label{eq:cost_J2}
\sum_{k = 0}^{N-1} (\|\bar{x}_{k} \|_Q^2 +\| \bar{u}_{k} \|_R^2) +\| \bar{x}_{N} \|_{P(\bar{\theta}_t)}^2 - \gamma \beta 
\end{equation}
where $\gamma\geq 0$ is a weight that controls the relative priority given to satisfaction of the PE condition (\ref{eq:PE_condition}) and tracking performance.  {This modification does not affect the feasibility of the optimisation. }

 {
Although incorporating a condition such as (\ref{eq:PE_condition}) or the convex relaxation (\ref{eq:PE_constraint})-(\ref{eq:cost_J2}) into a MPC strategy does not ensure that the closed-loop system satisfies a corresponding PE condition, its effect on the convergence rate of the estimated parameter set is significant, as shown in the numerical example in Section \ref{sec:example}.
}
 {
Also, $\gamma$ is a meaningful and straightforward coefficient to tune, 
and the PE constraint can be easily switched off by setting $\gamma = 0$.
}



\subsection{Proposed algorithm} \label{subsec:proposed_algorithm}
\textbf{Offline:}
\begin{enumerate}
\item Choose suitable $T$ defining the predicted state tube and compute the corresponding $U_j$ in (\ref{eq:define_U}).
\item Obtain a nominal $\bar{\theta}_0$. 
\item 
Minimise the contracitivity factor $\lambda$ satisfying (\ref{eq:lambda}) and obtain a feedback gain $K$.
\end{enumerate}

\noindent\textbf{Online:} For $t = 0,1,2,\dots$
\begin{enumerate}
\item Obtain the current state $x_t$ and set $x=x_t$.
\item Update $\Theta_{t} = \{\theta: M_{\Theta} \theta \leq \mu_{t} \}$ using (\ref{eq:fixed_set_update}) and the nominal parameter vector $\bar{\theta}_t$ using (\ref{eq:nominal_theta_update}), and solve (\ref{eq:lyap_mat}) for $P(\bar\theta_t)$.
\item 
At $t = 0$ compute the initial reference state and control sequences $\mathbf{\hat{x}} = \{\hat{x}_0,\dots, \hat{x}_{N} \}$ and $\hat{\mathbf{u}} = \{\hat{u}_0,\dots, \hat{u}_{N-1} \}$, for example by solving the nominal problem described in Remark \ref{rem:ref_initial}.
\\
At $t > 0$ 
compute the reference sequences $\mathbf{\hat{x}} $ and $\hat{\mathbf{u}} $, using the solution $\mathbf{v}_{t-1}^\ast=\{v^\ast_{0|t-1} , \ldots, v^\ast_{N-1|t-1}\}$ at $t-1$, and 
\begin{equation}
\begin{aligned}
&\hat{x}_0 = x\\
& \hat{x}_{k+1} = A(\bar{\theta}_{t})\hat{x}_k +B(\bar{\theta}_t)\hat{u}_k\\
& \hat{u}_{k} = K\hat{x}_{k}+\hat{v}_k \\
&\hat{v}_k = \begin{cases} v^\ast_{k+1|t-1} & k = 0, \dots, N-2\\ 0. & k = N-1\end{cases}
\end{aligned}
\end{equation}
\item 
Compute $\mathbf{v}^\ast_t = \{v_0^\ast,\dots, v_{N-1}^\ast \}$, $\boldsymbol{\alpha}^\ast_t = \{\alpha^\ast_0,\dots, \alpha^\ast_{N} \}$, $\bar{\mathbf{x}}^\ast_t = \{\bar{x}_0 ^\ast,\dots, \bar{x}_{N}^\ast \}$, $\bar{\mathbf{u}}^\ast_t = \{\bar{u}_0 ^\ast,\dots, \bar{u}_{N-1}^\ast \}$, $\beta^\ast_t$, {$\boldsymbol{\Lambda}^\ast = \{\Lambda_{k,j}, \ \forall\ k\in \NN_{[0,\dots,N]}, j\in \NN_{[1,\dots,p]} \}$} the solution of the semidefinite program (SDP):
\begin{equation*}
\mathcal{P}: 
\min_{\mathbf{v}, \boldsymbol{\alpha}, \bar{\mathbf{x}}, \bar{\mathbf{u}}, \beta {, \boldsymbol{\Lambda}}} \sum_{k = 0}^{N-1} (\|\bar{x}_{k} \|_Q^2 +\| \bar{u}_{k} \|_R^2) +\| \bar{x}_{N} \|_{P(\bar\theta_t)}^2 - \gamma \beta
\ \ \text{subject to (\ref{eq:initial_con}), (\ref{eq:input_state_constraint}),  (\ref{eq:update_constraint}),  (\ref{eq:final_tube_condition}), (\ref{eq:nominal_update}),  (\ref{eq:PE_constraint}). }
\end{equation*}
\item Implement the current control input $u_t = Kx_t +v_{0}^\ast$. 
\end{enumerate}
 {
\vspace{-5mm}
\begin{remark} \label{rem:T}
In \blue{offline} step 1, $T$ can be chosen so that the polytope $\X=\{x: Tx \leq \ones \}$ in Assumption~\ref{ass:lambda_contractive} approximates a Robust Control Invariant (RCI) set of the form $\{x: x^{\top} \tilde{P} x \leq\ones\}$, where $\tilde{P}=\tilde{P}^\top\succ 0$ satisfies $\tilde{P} - (A(\theta)+B(\theta)\tilde{K})^\top \tilde{P} (A(\theta)+B(\theta)\tilde{K}) \succ 0$ for some $\tilde{K}\in\mathbb{R}^{n_u\times n_x}$. 
Using the vertex representation of $\Theta_0$, 
the matrix $\tilde{P}$ can be computed by solving a semidefinite program~\cite[Chap.~5]{boyd94}. This approach allows the number of rows in $T$ to be specified by the designer. However, $T$ can alternatively be chosen so that $\X$ approximates the minimal Robust Positive Invariant (RPI) set or the maximal RPI set for the system (\ref{eq:update_equation})-(\ref{eq:constraints})
under a specified stabilizing feedback law~\cite{Blanchini}.
\end{remark}}
 {
\vspace{-5mm}
\begin{remark}
In \blue{offline} step 3, the computation of $\min_K \lambda$ subject to  (\ref{eq:lambda}) for given $T$ can be performed by solving a LP using the vertex representation of $\Theta_0$~\cite[Chap.~7]{Blanchini}. The objective of minimizing $\lambda$ is chosen to make the constraints of problem~$\mathcal{P}$ easier to satisfy. In particular, choosing $K$ so that $\lambda< 1$ in (\ref{eq:lambda}) ensures that $\alpha_N$ exists satisfying the terminal constraints (\ref{eq:final_tube_condition}b-d).
\end{remark}
}
\begin{remark} \label{rem:ref_initial}
At $t=0$, the reference sequences $\mathbf{\hat{x}} = \{\hat{x}_0,\dots, \hat{x}_{N} \}$ and $\hat{\mathbf{u}} = \{\hat{u}_0,\dots, \hat{u}_{N-1} \}$ may be computed by solving
\begin{equation} \label{eq:opt_algorithm1}
\min_{\mathbf{\hat{v}}}  \sum_{k = 0}^{N-1} (\|\hat{x}_{k} \|_Q^2 +\| \hat{u}_{k} \|_R^2) +\| \hat{x}_{N} \|_{P(\bar\theta_0)}^2 
\ \ \text{subject to } \ \ 
\begin{aligned}[t]
&\hat{x}_0 = x_0\\
&\hat{x}_{k+1} = A(\bar{\theta}_0) \hat{x}_{k} +B(\bar{\theta}_0) \hat{u}_{k}  \\
&\hat{u}_{k} = K \hat{x}_{k}+\hat{v}_{k}\\
&F\hat{x}_{k} +G\hat{u}_{k}\leq 1 
\end{aligned}
\end{equation}
\end{remark}

\begin{remark}\label{rem:periodic_update}
The online computation of the proposed algorithm may be reduced by updating $\Theta_t$ only once every $N_u>1$ time steps. For example, in Step 2, set $\Theta_t = \bigcap_{j=t-N_u+1}^{t} \Delta_j \cap \Theta_{t-N_u}$ for $t\in\{N_u,2N_u,\ldots\}$ and $\Theta_t = \Theta_{t-1}$ at all times $t\notin\{N_u,2N_u,\ldots\}$. 
\end{remark}


In Section~\ref{sec:feasibility_stability} we use the property that $\Theta_{t} \subseteq \Theta_{t-1}$ to show that the solution, $\mathbf{v}^\ast_{t-1}$, of $\mathcal{P}$ at time $t-1$ forms part of a feasible solution of $\mathcal{P}$ at time $t$. As a result, the reference sequences $\mathbf{\hat{x}}$, $\mathbf{\hat{u}}$ in Step 3 are feasible for problem $\mathcal{P}$ at all times $t>0$.

\section{Recursive Feasibility and Stability} \label{sec:feasibility_stability}
\subsection{Recursive feasibility}
At time $t\geq 1$, let a suboptimal set of decision variables, denoted $(\hat{\mathbf{v}}_t,\hat{\boldsymbol{\alpha}}_t)$ be defined in terms of the optimal solution $(\mathbf{v}^\ast_{t-1}, \boldsymbol{\alpha}^\ast_{t-1})$ of $\mathcal{P}$ at time $t-1$ by
\[
\hat{\mathbf{v}}_t = \{v^\ast_{1|t-1}, \dots, v^\ast_{N-1|t-1}, 0 \}, 
\quad
\hat{\boldsymbol{\alpha}}_t = \{\alpha^\ast_{1|t-1}, \dots, \alpha^\ast_{N|t-1}, \alpha^\ast_{N|t-1}\}.
\]

\begin{proposition} [Recursive Feasibility] \label{prop:recursive}
The online MPC optimization $\mathcal{P}$ is feasible at all times $t  \in \NN_{> 0}$ if $\mathcal{P}$ is feasible at $t = 0$ and $\Theta_{t}\subseteq \Theta_{t-1}$ for all time $t$.
\end{proposition}
\begin{proof}
If $\mathcal{P}$ is feasible at $t -1$, then at time $t$,  $(\mathbf{v},\boldsymbol{\alpha}) = (\hat{\mathbf{v}}_t,\hat{\boldsymbol{\alpha}}_t)$ is:
feasible for (\ref{eq:initial_con}) because $x_t \in \X_{1|t-1}$; 
feasible for (\ref{eq:input_state_constraint}) and (\ref{eq:update_constraint}) for $k \in \NN_{[0, N-2]}$ because $(\mathbf{v},\boldsymbol{\alpha}) = (\mathbf{v}^\ast_{t-1},\boldsymbol{\alpha}^\ast_{t-1})$ is feasible for  (\ref{eq:input_state_constraint}) and (\ref{eq:update_constraint}) for $k \in \NN_{[1, N-1]}$ and $\Theta_{t}\subseteq \Theta_{t-1}$;
and feasible for (\ref{eq:input_state_constraint}) and (\ref{eq:update_constraint}) for $k = N-1$ and feasible for (\ref{eq:final_tube_condition})
because $\alpha_N= \alpha^\ast_{N|t-1}$ is feasible for (\ref{eq:final_tube_condition}) and $\Theta_{t}\subseteq \Theta_{t-1}$.
%
%
Finally, we note that (\ref{eq:nominal_update}) is necessarily feasible and (\ref{eq:PE_constraint}) necessarily holds for some scalar $\beta$ if $(\mathbf{v},\boldsymbol{\alpha}) = (\hat{\mathbf{v}}_t,\hat{\boldsymbol{\alpha}}_t)$.
%
\end{proof}

\subsection{Input-to-state stability (ISS)}
Throughout this section we set $\gamma = 0$ in problem $\mathcal{P}$. Therefore the objective of $\mathcal{P}$ is $J(x, \mathbf{v}, \bar{\theta}_t)$ where $J$ is the nominal cost (\ref{eq:nominal_dual_mode_cost}). \blue{As a result of parameter adaption, the change of $\bar{\theta}_t$ online might increase the cost, but this is absorbed in the ISS terms.}
%
To simplify notation we define a stage cost $L(x,v)$ and terminal cost $\phi(x,\theta)$ as $L(x,v)  = \|{x} \|_Q^2 +\| {Kx+v}\|_R^2$ and $\phi(x,\theta) = \|x \|^2_{P(\theta)}$ so that (\ref{eq:nominal_dual_mode_cost}) is equivalent to
\[
J(x,\mathbf{v},\bar{\theta}_t) = \sum_{k = 0}^{N-1} L(\bar{x}_{k}, \bar{u}_{k}) + \phi(\bar{x}_{N},\bar{\theta}_t).
\]
%
Denoting the actual state at next time step as $x^+$, we define the function $ f(x,v,w,\theta^\ast)$ as
\[
x^+ = f(x,v,w,\theta^\ast)=  \bigl(A(\theta^\ast) +B(\theta^\ast)K\bigr) x + B(\theta^\ast)v + w,
\]
so that $x_{t+1} = f(x_t,v_t,w_t,\theta^\ast)$.



\begin{lemma}[ISS-Lyapunov function~\cite{Limon2009}] \label{lem:ISS_condition}
The system
\begin{equation} \label{eq:update_nonlinear}
x^+ = f(x,v,w,\theta),
\end{equation}
with control law $v=v(x,\btheta_t,t)$ is ISS with region of attraction $ {\mathcal{R} \subseteq \RR^{n_x}}$ if the following conditions are satisfied.\\
(i).~
$ {\mathcal{R}}$ contains the origin in its interior, \red{is compact,} and is a robust positively invariant set for (\ref{eq:update_nonlinear}), i.e. $f\bigl(x,v(x,\btheta_t,t),w,\theta\bigr) \in  {\mathcal{R}}$ for all $x\in  {\mathcal{R}}$, $w\in \mathcal{W}$, $ \theta \in \Theta$ and $t\in\NN_{\geq 0}$. \\
(ii).~
There exist $\mathcal{K}_\infty$- functions $ \varsigma_1, \varsigma_2,  \varsigma_3$, $\mathcal{K}$-functions $\sigma_1$, $\sigma_2$ and a 
function $\mathcal{V}:  {\mathcal{R}} \times \NN_{\geq 0} \rightarrow \RR_{\geq 0}$ such that for all $t\in\NN_{\geq 0}$, $\mathcal{V}(\cdot,t)$ is continuous, and for all $(x,t)\in {\mathcal{R}}\times\NN_{\geq 0}$,%
\begin{gather}
\varsigma_1(\|x \|) \leq \mathcal{V}(x,t) \leq  \varsigma_2(\| x\|)\label{eq:ISS_condition_1}, \\
\mathcal{V}(x^+,t+1)-\mathcal{V}(x,t) \leq - \varsigma_3(\|x\|)+\sigma_1(\|w\|) +\sigma_2(\| \bar{\theta}_t - \theta^\ast \|) \label{eq:ISS_condition_2}.
\end{gather}
\end{lemma}


In the following we define $X_\mathcal{P}$ as the set of states $x$ such that problem $\mathcal{P}$ is feasible and assume that $X_\mathcal{P}$ is non-empty. 
In addition, for a given state $x$, nominal parameter vector $\bar{\theta}$ and parameter set $\Theta$, we denote $\mathbf{v}^\ast(x,\bar{\theta},\Theta)$ as the optimal solution of problem $\mathcal{P}$, and let $V^\ast(x,\bar{\theta},\Theta)$ be the corresponding optimal value of the cost in (\ref{eq:nominal_dual_mode_cost}), so that $V^\ast(x,\bar{\theta},\Theta) = J(x, \mathbf{v}^\ast(x,\bar{\theta},\Theta),\bar{\theta})$.

\begin{theorem} \label{thm:ISS_theta}
Assume that $\gamma = 0$ and the nominal parameter vector $\bar{\theta}_t$ 
is not updated, i.e.~$\bar{\theta}_t = \bar{\theta}_0$ for all $t\in\NN_{\geq 0}$.
Then for all initial conditions $x_0 \in X_\mathcal{P}$, the system (\ref{eq:update_equation}) with control law $u_t = Kx_t + v_{0|t}^\ast$, where $v_{0|t}^\ast$ is the first element of $\mathbf{v}^\ast(x_t, \btheta_t,\Theta_t)$,
robustly satisfies the  constraint (\ref{eq:constraints})
and is ISS with region of attraction $X_\mathcal{P}$.
\end{theorem}

\begin{proof}
We first show that condition (i) of Lemma~\ref{lem:ISS_condition} is satisfied with
$ {\mathcal{R}} = X_\mathcal{P}$.
If $\mathcal{P}$ is feasible at $t=0$, then (\ref{eq:final_tube_condition}) implies that $\alpha_N$ exists such that $\X_N = \{x : Tx\leq \alpha_N \}$ satisfies
\begin{equation}\label{eq:feasible_term_set}
\begin{gathered}
\bigl( A(\theta) + B(\theta)K\bigr) x + w \in \X_N \ \ \forall x\in\X_N,\ w\in\W, \ \theta\in\Theta_0, \\
(F+GK)x \leq \ones \ \ \forall x\in\X_N .
\end{gathered}
\end{equation}
Therefore $(\mathbf{v},\boldsymbol{\alpha}) = (\{0,\ldots,0\}, \{\alpha_N,\ldots,\alpha_N\})$ is feasible for $\mathcal{P}$ for all $x_0\in\X_N$, and hence $\X_N \subseteq  X_{\mathcal{P}}$. 
In addition, the robust invariance of $\X_N$ implied by (\ref{eq:feasible_term_set}) ensures that $\W\subseteq\X_N$, and since $0\in\interior(\W)$ due to Assumption~\ref{ass:compact_disturbance}, $X_{\mathcal{P}}$ must contain the origin in its interior.
%
%
%
Furthermore, Proposition \ref{prop:recursive} shows that if $\mathcal{P}$ is initially feasible, then it is feasible for all $t \geq 0$, \red{so that $X_{\mathcal{P}}$ is positively invariant for  (\ref{eq:update_nonlinear}). Finally, $X_{\mathcal{P}}$ is necessarily compact by Assumption~\ref{ass:compact_feasible_region},} and it follows that condition (i) of Lemma~\ref{lem:ISS_condition} is satisfied if $ {\mathcal{R}} = X_\mathcal{P}$. 

We next consider the bounds (\ref{eq:ISS_condition_1}) in condition (ii) of Lemma~\ref{lem:ISS_condition}. 
For a given state $x$, nominal parameter vector $\bar{\theta}$ and parameter set $\Theta$, 
%
problem $\mathcal{P}$ with $\gamma=0$ and $Q,R\succ 0$ is a convex quadratic program. 
Therefore 
$V^\ast(x,\bar{\theta},\Theta)$ is a continuous positive definite, piecewise quadratic function of $x$ \cite{bemporad2002} for each $\btheta\in\Theta_0$ and $\Theta\subseteq\Theta_0$. 
Furthermore $\Theta_0$ is compact due to Assumption~\ref{ass:param_uncertainty} and it follows that 
there exist $\mathcal{K}_\infty$-functions $ \varsigma_1$, $ \varsigma_2$ such that (\ref{eq:ISS_condition_1}) holds with 
\[
\mathcal{V}(x,t) = V^\ast(x,\bar{\theta}_t,\Theta_t).
\]

To show that the bound (\ref{eq:ISS_condition_2}) in condition (ii) of Lemma~\ref{lem:ISS_condition} holds,
let $\mathcal{F}_{\mathcal{P}}$ denote the set $\{v : (x, Kx+v) \in \mathcal{Z},\ x\in X_{\mathcal{P}}\}$. Then, given the linear dependence of the system (\ref{eq:update_equation}), the model parameterisation (\ref{eq:system_model}) and the predicted control law (\ref{eq:input_law}) on the state $x$, disturbance $w$ and parameter vector $\theta$, and since $\W$, $\Theta_0$, $X_{\mathcal{P}}$ and $\mathcal{F}_{\mathcal{P}}$ are compact sets by Assumptions~\ref{ass:compact_disturbance}, \ref{ass:param_uncertainty} and \ref{ass:compact_feasible_region}, there exist $\mathcal{K}_\infty$ functions $\sigma_{x}$, $\sigma_w$, $\sigma_\theta$, $\sigma_L$, $\sigma_\phi$ such that, $\forall x,x_1,x_2\in X_{\mathcal{P}}$, $\forall v\in \mathcal{F}_{\mathcal{P}}$, $\forall w,w_1,w_2\in \W$, $\forall \theta,\theta_1,\theta_2\in\Theta_0$,
\begin{align*}
\| f(x_1,v,w_1,\theta_1)-f(x_2,v,w_2,\theta_2)\|  &\leq \sigma_{x}(\|x_1-x_2\|) + \sigma_w(\|w_1-w_2\|) +\sigma_{\theta}(\|\theta_1-\theta_2\|) ,\\
\lvert L(x_1,v)-L(x_2,v)\rvert  &\leq \sigma_L(\|x_1-x_2\|) ,\\
\lvert \phi(x_1,\theta)-\phi(x_2,\theta)\rvert &\leq \sigma_{\phi}(\|x_1-x_2\|).
\end{align*}
%
 {Following the proof of Theorem 5 in Limon et al.~\cite{Limon2009} and using the weak triangle inequality for $\mathcal{K}$-functions \cite{Sontag1989}, we obtain}
\begin{equation*}
J(x^+,\hat{\mathbf{v}},\btheta_t)  - \mathcal{V}(x,t) \leq - \varsigma_3 (\|x \|) + \sigma_1(\|w\|) +\sigma_2(\| \btheta_t - \theta^\ast \|)
\end{equation*}
where $\sigma_1(\|s\|) = \big(\sum_{k = 0}^{N-2} \sigma_L\circ \sigma_x^k +\sigma_\phi \circ \sigma_x^{N-1} \big)\circ 2\sigma_w(\|s\|)$ and $\sigma_2(\|s \|)= \big(\sum_{k = 0}^{N-2} \sigma_L\circ \sigma_x^k +\sigma_\phi \circ \sigma_x^{N-1} \big)\circ 2\sigma_\theta((\|s \|))$, and both $\sigma_1$ and $\sigma_2$ are $\mathcal{K}$-functions. 
Since $\hat{\mathbf{v}}$ is a feasible but suboptimal solution of $\mathcal{P}$ at $x^+$, and since $\btheta_{t+1}=\btheta_t$ by assumption, the optimal cost function satisfies 
$\mathcal{V}(x^+,t+1)  = V^\ast(x^+,\btheta_{t+1},\Theta_{t+1}) \leq J(x^+,\hat{\mathbf{v}},\btheta_t)$ and hence
\[
\mathcal{V}(x^+,t+1) - \mathcal{V}(x,t) \leq - \varsigma_3(\|x \|)+\sigma_1(\|w \|)+ \sigma_2(\| \btheta_t - \theta^\ast \|) .
\]
Thus all conditions of lemma \ref{lem:ISS_condition} are satisfied.
\end{proof}

\begin{corollary} \label{cor:ISS_theta}
Assume that $\gamma = 0$ and the nominal parameter vector $\bar{\theta}_t$ 
is updated at each time $t\in\NN_{\geq 0}$ using (\ref{eq:nominal_theta_update}). 
Then for all initial conditions $x_0 \in X_\mathcal{P}$, the system (\ref{eq:update_equation}) with control law $u_t = Kx_t + v_{0|t}^\ast$, where $v_{0|t}^\ast$ is the first element of $\mathbf{v}^\ast(x_t, \btheta_t,\Theta_t)$,
robustly satisfies the constraint (\ref{eq:constraints})
and is ISS with region of attraction $X_\mathcal{P}$.
\end{corollary}

\begin{proof}
It can be shown that condition (i) 
of Lemma~\ref{lem:ISS_condition} 
and the bounds (\ref{eq:ISS_condition_1}) 
in condition (ii) of Lemma~\ref{lem:ISS_condition} 
are satisfied with $ {\mathcal{R}} = X_{\mathcal{P}}$ and $\mathcal{V}(x,t) = V^\ast(x,\btheta_t,\Theta_t)$
for some $\mathcal{K}_\infty$-functions $ \varsigma_1$, $ \varsigma_2$
using the same argument as the proof of Theorem~\ref{thm:ISS_theta}.
To show that (\ref{eq:ISS_condition_2}) is also satisfied and hence complete the proof we use an argument similar to the proof of Theorem~\ref{thm:ISS_theta}.
In particular, as before we define $\bar{\mathbf{x}}^\ast=\{\bar{x}_0^\ast,\ldots,\bar{x}_N^\ast\}$ using the optimal solution of $\mathcal{P}$, $\mathbf{v}^\ast(x,\bar\theta_t,\Theta_t) = \{v_0^\ast,\ldots,v_{N-1}^\ast\}$, and
\[
\bar{x}_{k+1}^\ast = 
f(\bar{x}_k^\ast,v_k^\ast,0,\bar{\theta}_t),
\quad \bar{x}_0^\ast = x. 
\]
However, here we define $\mathbf{z} = \{z_0,\ldots,z_N\}$ as the sequence
\[
z_{k+1} = f(z_k,\hat{v}_k,0,\bar{\theta}_{t+1}),
\quad z_0 = x^+
\]
where $\hat{\mathbf{v}}=\{\hat{v}_0, \dots, \hat{v}_{N-1} \}$ has $\hat{v}_k = v_{k+1}^\ast$ for $k\in\mathbb{N}_{[0,N-2]}$ and $\hat{v}_{N-1} = 0$. Then
$\mathcal{V}(x,t) = V^\ast(x,\bar{\theta}_t,\Theta_t)$ implies
\begin{equation}\label{eq:cost_change}
\begin{aligned}
J(x^+,\hat{\mathbf{v}},\btheta_{t+1}) - \mathcal{V}(x,t) 
&= - L(x , v_0^\ast) + \sum_{k = 0}^{N-2} \bigl( L(z_k,\hat{v}_k) - L(\bar{x}_{k+1}^\ast , v^\ast_{k+1}) \bigr) 
+  L(z_{N-1}, 0) 
\\
&\quad 
+ \phi(z_N,\btheta_{t+1}) 
- \phi(z_{N-1},\btheta_{t+1})
+  \phi(z_{N-1},\btheta_{t+1}) 
- \phi(\bar{x}_N^\ast,\btheta_t) .
\end{aligned}
\end{equation}
The update law (\ref{eq:nominal_theta_update}) 
ensures that $\| \btheta_{t+1} - \btheta_t \| \leq \| \btheta_{t} - \theta^\ast\|$
since $\theta^\ast \in \Theta_{t+1}$. Hence, for all $k\in\NN_{[1,N-1]}$, we have
\begin{align*}
\|z_k - \bar{x}^\ast_{k+1}\| &\leq \sigma_x(\|z_{k-1} - \bar{x}^\ast_k\|) + \sigma_\theta(\|\btheta_{t+1}-\btheta_t\|) \\
& \leq 
\sigma_x(\|z_{k-1} - \bar{x}^\ast_k\|) + \sigma_\theta(\|\btheta_{t}-\theta^\ast\|) 
\end{align*}
and it follows that, for all $k\in\NN_{[1,N-1]}$,
\[
\|z_k - \bar{x}^\ast_{k+1}\| \leq
\tfrac{1}{2}(2\sigma_x)^k(\|x^+ - \bar{x}^\ast_1\|) + \sum_{j=0}^{k-1} \tfrac{1}{2} (2\sigma_x)^j\circ 2\sigma_\theta (\|\theta_t - \theta^\ast\|) 
\]
where $\|x^+ - \bar{x}_1^\ast\| \leq  \sigma_\theta (\| \btheta_t - \theta^\ast \|) + \sigma_w(\|w_0\|)$.
In addition we have, for all $k\in\NN_{[0,N-2]}$
\[
\lvert L(z_k,\hat{v}_k)-L(\bar{x}_{k+1}^\ast, v^\ast_{k+1}) \rvert  \leq \sigma_L ( \|z_k - \bar{x}_{k+1}^\ast \|),
\]
and, since $Q\succ 0$, there exists a $\mathcal{K}_\infty$ function $ \varsigma_3$ such that
$L(x,v_0^\ast) \geq  \varsigma_3 (\|x \|)$, while (\ref{eq:lyap_mat}) implies
\[
L(z_{N-1}, 0) + \phi(z_N,\btheta_{t+1}) - \phi(z_{N-1},\btheta_{t+1}) =0 .
\]
Furthermore (\ref{eq:lyap_mat}) is linear in $P(\theta)$ and the solution $P(\theta)$ is unique for all $\theta\in\Theta_t$ (see e.g.~\cite{khalil2002}) since $A(\theta)+B(\theta)K$ is by assumption stable. Therefore, by the implicit function theorem, $P(\theta)$ is Lipschitz continuous and $\lvert \phi(x,\theta_1) - \phi(x,\theta_2)\rvert \leq \kappa_\phi \| \theta_1 - \theta_2\|$ for all $x\in X_{\mathcal{P}}$, $\theta_1,\theta_2\in\Theta_0$, for some $\kappa_\phi > 0$. Hence
\[
\lvert \phi(z_{N-1},\bar{\theta}_{t+1}) - \phi(\bar{x}_N^\ast,\bar{\theta}_t)\rvert \leq
\sigma_\phi(\|z_{N-1} - \bar{x}_N^\ast\|) 
+ \kappa_\phi\|\btheta_{t+1} - \btheta_t\|.
\]
Collecting the bounds derived above on individual terms in the expression for $J(x^+,\hat{\mathbf{v}},\btheta_{t+1}) - \mathcal{V}(x,t)$ in (\ref{eq:cost_change}), we obtain
\begin{equation*}
J(x^+,\hat{\mathbf{v}},\btheta_{t+1})  - \mathcal{V}(x,t) \leq - \varsigma_3 (\|x \|) + \sigma_1(\|w\|) +\sigma_2(\| \btheta_t - \theta^\ast \|)
\end{equation*}
where $\sigma_1$, $\sigma_2$ are  $\mathcal{K}$-functions.
But by optimality we have $\mathcal{V}(x^+,t+1) = V^\ast(x^+,\bar{\theta}_{t+1},\Theta_{t+1}) \leq J(x^+,\hat{\mathbf{v}},\btheta_{t+1})$. Therefore (\ref{eq:ISS_condition_2}) holds and hence all of the conditions of Lemma~\ref{lem:ISS_condition} are satisfied.
\end{proof}

\begin{remark}
The input-to-state stability property implies that there exists a $\mathcal{K}\mathcal{L}$-function $\eta(\cdot,\cdot)$ and $\mathcal{K}$-functions $\psi(\cdot)$ and $\zeta(\cdot)$ such that for all feasible initial conditions $x_0\in X_{\mathcal{P}}$, the closed loop system trajectories satisfy, for all $t\in\NN_{\geq 0}$,
\[
\|x_t\| \leq \eta(\|x_0\|,t) + \psi\bigl(\max_{k\in\NN_{[0, t-1]}} \|w_k\|\bigr) + \zeta\bigl(\max_{k\in\NN_{[0, t-1]}} \|\btheta_k-\theta^\ast\|\bigr).
\]
\end{remark}
 {
\begin{remark}
Theorem \ref{thm:ISS_theta} and Corollary \ref{cor:ISS_theta} do not apply to the case that $\gamma \neq 0$. However, if $\gamma$ is replaced by a time-varying weight $\gamma_t$ in the objective of problem $\mathcal{P}$, then \blue{input-to-state stability (ISS)} can be guaranteed by switching $\gamma_t =0$ for all $t\geq t_0$, for some finite horizon $t_0$.
\end{remark}}

\section{Convergence of the estimated parameter set} \label{sec:convergence_parameter}

In terms of $D$ and $d$ defined in (\ref{eq:Dk}) and (\ref{eq:dk}), the system model $x_{t+1} = A(\theta^\ast) x_t+ B(\theta^\ast) u_t + w_t$ can be rewritten as
\begin{equation}\label{eq:model_sec5}
x_{t+1}  = D(x_t, u_t)\theta^\ast +d(x_t, u_t) +w_t
\end{equation}
where $x_{t+1}$, $D(x_t,u_t)$ and $d(x_t, u_t)$ are known at time $t+1$. Thus, the system is linear with regressor $D_t$, uncertain parameter vector $\theta^\ast$ and additive disturbance $w_t\in\W$. 

Bai et al.~\cite{Bai1998} show that, for such a system, the diameter of the parameter set constructed using a set-membership identification method converges to zero with probability 1 if the uncertainty bound $\W$ is tight and the regressor $D_t$ is persistently exciting. We extend this result and prove convergence of the estimated parameter set in more general cases. Specifically, in this paper we avoid the problem of computational intractability arising from a minimum volume update law of the form $\Theta_{t+1}  = \Theta_t \cap  { \Delta_{t+1}}$. Instead, we derive stochastic convergence results for parameter sets with fixed complexity and update laws of the form $\Theta_{t+1}  \supseteq \Theta_t \cap  { \Delta_{t+1}}$. 

In this section we first discuss relevant results for an update law that gives a minimal parameter set estimate for a given sequence of states (but whose representation has potentially unbounded complexity), before considering convergence of the fixed-complexity parameter set update law of Section~\ref{sec:fixed_comp_set_update}. 
We then compute bounds on the parameter set diameter if the bounding set for the additive disturbances is overestimated. Lastly, we demonstrate that similar results can be achieved when errors are present in the observed state, as would be encountered for example if the system state were estimated from noisy measurements.
 {In each case we relate the PE condition to the rate of parameter set convergence. We also prove that the parameter set converges to a point (or minimal uncertainty set) with probability one.}

In common with Bai et al.~\cite{Bai1998,Bai1995}, we do not assume a specific distribution for the disturbance input. However, the set $\W$ bounding the model disturbance is assumed to be tight in the sense that there is non-zero probability of a realisation $w_t$ lying arbitrarily close to any given point on the boundary, $\partial\W$, of $\W$.

\begin{assumption}[Tight disturbance bounds]\label{ass:probw}
For all $w^0\in \partial\W$ and any $\epsilon > 0$ the disturbance sequence $\{w_0,w_1,\ldots\}$ satisfies
$\Pr \bigl\{ \| w_t - w^0 \| < \epsilon \bigr\} \geq \probw (\epsilon) $,
for all $t\in\NN_{\geq 0}$, where $\probw(\epsilon) > 0$ whenever $\epsilon > 0$.
\end{assumption}

\begin{assumption}[Persistent Excitation]\label{ass:pe}
There exist positive scalars $\tau$, $\beta$ and an integer $N_u\geq \lceil p/n_x \rceil$ such that, for each $t\in\NN_{\geq 0}$ we have $\| D_t\| \leq \tau$ and
\[
 {
\sum_{j=t}^{t+N_u-1 } D_j^\top D_j \succeq \beta \Id.}
\]
\end{assumption}

\noindent We further assume throughout this section that the rows of $M_\Theta$ are normalised so that $\|[M_\Theta]_i\| = 1$ for all $i$.

\subsection{Minimal parameter set} \label{sec:min_set}
The unfalsified parameter set at time $t$ defined in (\ref{eq:delta_t_set}) can be expressed as
\begin{equation}\label{eq:Delta}
\Delta_{t} = \{\theta : D_{t-1} (\theta^\ast - \theta) + w_{t-1} \in \W\} ,
\end{equation}
where $w_t$ is the disturbance realisation at time $t$ and $D_t = D(x_t,u_t)$.
Let $w^0$ be an arbitrary point on the boundary $\partial \W$, then the normal cone $\N_\W(w^0)$ to $\W$ at $w^0$ is defined 
\begin{equation}\label{eq:normal_cone}
\N_\W(w^0) \defeq \{ g : g^\top (w - w^0) \leq 0 \ \ \forall w \in\W \} .
\end{equation}
%

\begin{proposition}\label{prop:min_set}
For all $t\in\NN_{\geq 0}$, all $\epsilon > 0$, and for any $\theta\in\mathbb{R}^p$ such that $\|\theta^\ast - \theta \| \geq \epsilon$, under Assumptions~\ref{ass:compact_disturbance}, \ref{ass:probw} and \ref{ass:pe} we have
\[
\Pr\{ \theta \not\in\Delta_j\} \geq \probw\bigl(\epsilon\sqrt{\beta/N_u}\bigr)
\]
for some $j\in\NN_{[t+1,t+N_u]}$.
\end{proposition}


\begin{proof}
Assumption~\ref{ass:compact_disturbance} implies that there exists $w^0\in\partial\W$ so that $D_j (\theta^\ast - \theta) \in \N_\W(w^0)$ for any given $j\in\NN_{[t,t+N_u-1]}$ and $\theta\in\Theta_t$ 
Therefore, if $\theta$ satisfies $(\theta^\ast - \theta)^\top D_j^\top [D_j (\theta^\ast - \theta) + w_j - w^0] > 0$, then the definition (\ref{eq:normal_cone}) of $\N_\W(w^0)$ implies $D_j (\theta^\ast - \theta) + w_j \notin \W$, and hence $\theta\notin\Delta_{j+1}$ from (\ref{eq:Delta}).
But
\[
(\theta^\ast - \theta)^{\!\top} D_j^{\!\top} \bigl[D_j (\theta^\ast - \theta) +  w_j - w^0\bigr] 
= \|D_j (\theta^\ast - \theta)\|^2 + (\theta^\ast - \theta)^\top D_j^\top(w_j- w^0) 
\geq  \|D_j (\theta^\ast - \theta)\|^2 - \|D_j (\theta^\ast - \theta)\| \, \|w_j  - w^0\|.
\]
Therefore $\theta\notin\Delta_{j+1}$ whenever $\|w_j - w^0\| < \|D_j (\theta^\ast - \theta)\|$. Furthermore, for all $t\in\NN_{\geq 0}$ Assumption~\ref{ass:pe} implies
\[
\sum_{j=t}^{t+N_u-1} \|D_j(\theta^\ast-\theta)\|^2 \geq \beta \|\theta^\ast - \theta\|^2 .
\]
Hence, if $\|\theta^\ast - \theta \| \geq \epsilon$, then there must exist some $j\in\NN_{[t,t+N_u-1]}$ such that
\[
\| D_j (\theta^\ast - \theta) \| \geq \epsilon \sqrt{\beta/N_u}.
\]
If $\|w_j - w^0\| < \epsilon\sqrt{\beta/N_u}$, then it follows that $\|w_j - w^0\| < \epsilon\sqrt{\beta/N_u} \leq \|D_j(\theta^\ast - \theta)\|$ and thus $\theta \notin \Delta_{j+1}$. Assumption~\ref{ass:probw} implies the probability of this event is at least $\probw\bigl(\epsilon\sqrt{\beta/N_u}\bigr)$. 
\end{proof}

\begin{theorem}\label{thm:min_set}
If $\Theta_t = \bigcap_{j=1}^{t} \Delta_j \cap \Theta_0$ and Assumptions~\ref{ass:compact_disturbance}, \ref{ass:probw} and \ref{ass:pe} hold, then for all $\theta\in\Theta_0$ such that $\|\theta - \theta^\ast \| \geq \epsilon$, for all $t\in\NN_{\geq 0}$ and any $\epsilon >0$, we have
\[
\Pr \{ \theta \in \Theta_t \}  \leq \biggl[ 1 - \probw \bigl(\epsilon \sqrt{\beta/N_u} \bigr) \biggr]^{\lfloor t/N_u \rfloor}  .
\]
\end{theorem}

\begin{proof}
For the non-trivial case of $t\geq N_u$ we have $\Pr \{\theta\in\Theta_t\} = \Pr \{\theta\in\Theta_{t} \, | \, \theta\in\Theta_{t-N_u} \} \Pr\{\theta\in\Theta_{t-N_u}\}$ since $\Theta_t\subseteq\Theta_{t-N_u}$. Also $\Theta_t = \bigcap_{j=t-N_u+1}^{t} \Delta_j \cap \Theta_{t-N_u}$ and Proposition~\ref{prop:min_set} implies
$
\Pr \{\theta\in\Theta_{t} \, | \, \theta\in\Theta_{t-N_u} \} \leq  1 - \probw \bigl(\epsilon \sqrt{\beta/N_u}\bigr)
$
if  $\|\theta - \theta^\ast\| \geq \epsilon$. Therefore 
\[
\Pr \{\theta\in\Theta_t\} \leq \Bigl( 1 - \probw \bigl(\epsilon \sqrt{\beta/N_u}\bigr)\Bigr)\Pr\{\theta\in\Theta_{t-N_u}\},
\]
and the result follows by applying this inequality $\lfloor t/N_u\rfloor$ times. 
\end{proof}

\begin{corollary}\label{cor:min_set}
Under Assumptions~\ref{ass:compact_disturbance}, \ref{ass:probw} and \ref{ass:pe},
$\Theta_t = \bigcap_{j=1}^t\Delta_j\cap\Theta_0$ converges to $\{\theta^\ast\}$ with probability~1.
\end{corollary}

\begin{proof}
For any $\theta\in\Theta_0$ and $\epsilon > 0$ such that $\|\theta- \theta^\ast\|\geq \epsilon$,
Theorem~\ref{thm:min_set} implies that $\sum_{t=0}^\infty \Pr\{\theta\in\Theta_t\}$ is necessarily finite, and since $\theta\in\Theta_t$ requires that $\theta\in\Theta_{t-1}$, the Borel-Cantelli Lemma therefore implies that $\Pr\bigl\{\theta\in\bigcap_{t=0}^\infty\Theta_t \bigr\} = 0$. It follows that $\Theta_t\to\{\theta^\ast\}$ as $t\to\infty$ with probability 1 since $\epsilon > 0$ is arbitrary.
\end{proof}


\subsection{Fixed complexity parameter set} \label{sec:fixed_set}
In order to reduce computational load and ensure numerical tractability, we assume that the parameter set $\Theta_t$ is defined by a fixed complexity polytope, as in (\ref{eq:defn_Theta_set}) and (\ref{eq:fixed_set_update}). This section shows that, although a degree of conservativeness is introduced by fixing the complexity of $\Theta_t$, asymptotic convergence of this set to the true parameter vector $\theta^\ast$ still holds with probability 1. 

\begin{theorem}\label{thm:fixed_set}
If $\Theta_t$ is updated according to (\ref{eq:defn_Theta_set}), (\ref{eq:fixed_set_update}) and Remark~\ref{rem:periodic_update}, and Assumptions \ref{ass:compact_disturbance}, \ref{ass:probw} and \ref{ass:pe} hold, then for all $\theta\in\Theta_0$ such that $[M_\Theta]_i (\theta- \theta^\ast) \geq \epsilon$ for some $i\in\NN_{[1,r]}$, we have, for all $t\in\NN_{\geq 0}$ and any  $\epsilon > 0$,
\[
\Pr \{ \theta\in\Theta_t \} \leq 
\biggl\{ 1 - \biggl[ \probw \Bigl( \frac{\epsilon\beta}{N_u\tau} \Bigr) \biggr]^{N_u}\biggr\}^{\lfloor t/N_u \rfloor} .
\]
\end{theorem}

\begin{proof}
For the non-trivial case of $t\geq N_u$ we have $\Pr \{\theta\in\Theta_t\} = \Pr \{\theta\in\Theta_{t} \, | \, \theta\in\Theta_{t-N_u} \} \Pr\{\theta\in\Theta_{t-N_u}\}$ since $\Theta_t\subseteq\Theta_{t-N_u}$ by Lemma~\ref{lem:monotonic_param_set}.
Consider therefore the probability that any given $\theta\in\Theta_{t-N_u}$ satisfying $[M_\Theta]_i (\theta- \theta^\ast) \geq \epsilon$ lies in $\Delta_{t-N_u+1}\cap\cdots\cap\Delta_t$.
Define vectors $g_j$ for $j\in\NN_{[t-N_u,t-1]}$ by
\begin{equation}\label{eq:normal_def}
g_j^\top = -[M_\Theta]_i \biggl(\sum_{k=t-N_u}^{t-1} D_k^\top D_k \biggr)^{-1} D_j^\top .
\end{equation}
%
Assumption \ref{ass:compact_disturbance} implies that, for any given $g_j\in\RR^{n_x}$, there exists a $w_j^0\in\partial\W$ such that $g_j \in\N_\W(w_j^0)$. Accordingly, choose $w_j^0\in\partial\W$ so that $g_j$ in (\ref{eq:normal_def}) satisfies $g_j \in\N_\W(w_j^0)$ for each $j\in\NN_{[t-N_u,t-1]}$. Then 
\begin{equation}\label{eq:nec_normal_def}
g_j^\top \bigl[ D_j (\theta^\ast - \theta) + w_j - w_j^0 \bigr] \leq 0 
\end{equation}
is a necessary condition for $\theta\in\Delta_{j+1}$ due to (\ref{eq:Delta}) and (\ref{eq:normal_cone}). But (\ref{eq:normal_def}) and Assumption~\ref{ass:pe} imply
\[
\sum_{j=t-N_u}^{t-1} g_j^\top\bigl[ D_j(\theta^\ast - \theta) + w_j - w_j^0\bigr] 
= [M_\Theta]_i(\theta - \theta^\ast) + \sum_{j=t-N_u}^{t-1} g_j^\top (w_j-w_j^0)
\geq 
[M_\Theta]_i(\theta - \theta^\ast) - \frac{N_u\tau}{\beta} \max_{j\in\NN_{[t-N_u,t-1]}} \| w_j - w_j^0\|. 
\]
where $[M_\Theta]_i(\theta - \theta^\ast) \geq \epsilon$ 
by assumption, and it follows from (\ref{eq:nec_normal_def}) that $\theta\not\in \bigcap_{j=t-N_u+1}^t \Delta_j$ if $\|w_j - w_j^0\| < \epsilon\beta/(N_u\tau)$ for all $j\in\NN_{[t-N_u,t-1]}$. From Assumption~\ref{ass:probw} and the independence of the sequence $\{w_0,w_1,\ldots\}$ we therefore conclude that
\[
\Pr \{\theta\in\Theta_t\} \leq \biggl\{ 1 - 
\biggl[ \probw \Bigl( \frac{\epsilon\beta}{N_u\tau} \Bigr) \biggr]^{N_u}\biggr\} 
\Pr \{\theta\in\Theta_{t-N_u}\},
\]
and the result follows by applying this inequality $\lfloor t/N_u\rfloor$ times.
\end{proof}

\begin{corollary}\label{cor:fixed_set}
Under Assumptions~\ref{ass:compact_disturbance}, \ref{ass:probw} and \ref{ass:pe}, the fixed complexity parameter set estimate $\Theta_t$ converges to $\{\theta^\ast\}$ with probability~1.
\end{corollary}

\begin{proof}
By applying the Borel-Cantelli Lemma to  {Theorem~\ref{thm:fixed_set}} it can be shown (analogously to the proof of Corollary~\ref{cor:min_set}) that $\Pr\{\theta\in\bigcap_{t=0}^\infty \Theta_t \} = 0$ if $[M_\Theta]_i (\theta- \theta^\ast) \geq \epsilon$ for some $i\in\NN_{[1,r]}$ and $\epsilon > 0$. Since $M_\Theta$ is assumed to be chosen so that $\Theta_t$ is compact for all $\mu_t$ such that $\Theta_t$ is non-empty, it follows that $\Theta_t\to\{\theta^\ast\}$ as $t\to\infty$ with probability 1.
\end{proof}

\subsection{Inexact disturbance bounds} \label{sec:inexact_noise_bound}
We next consider the case in which the set $\W$ bounding $w_t$ in Assumption~\ref{ass:compact_disturbance} does not satisfy Assumption~\ref{ass:probw}.
Instead, we assume that a compact set $\Wtight$ providing a tight bound on $w_t$ exists but is either unknown or non-polytopic or nonconvex.
We define the unit ball $\B = \{x : \|x\| \leq 1\}$ and use a scalar $\rho$ to characterize the accuracy to which $\W$ approximates $\Wtight$.

\begin{assumption}[Inexact disturbance bounds]\label{ass:hW}
$\Wtight$ is a compact set such that $\Wtight\oplus \rho\B \supseteq \W \supseteq \Wtight$ for some $\rho > 0$, and, for all $w^0\in \partial\Wtight$ and $\epsilon > 0$, the disturbance sequence $\{w_0,w_1,\ldots\}$ satisfies, for all $t\in\NN_{\geq 0}$, $w_t\in\Wtight$ and
$\Pr \bigl\{ \| w_t - w^0 \| < \epsilon \bigr\} \geq \probw (\epsilon)$,
where $\probw(\epsilon) > 0$ whenever $\epsilon > 0$.
\end{assumption}

\begin{remark}
Assumption~\ref{ass:hW} implies that 
$\W\ominus\Wtight \subseteq \rho\B$. As a result, every point in $\W$ can be a distance no greater than $\rho$ from a point in $\Wtight$, i.e.\ 
$
\max_{\hw\in\W}\min_{w\in\Wtight}\|\hw - w\| \leq \rho 
$.
\end{remark}


\begin{theorem}\label{thm:min_set_hat}
If $\Theta_t = \bigcap_{j=1}^{t} \Delta_j \cap \Theta_0$ and Assumptions~\ref{ass:compact_disturbance}, \ref{ass:pe} and~\ref{ass:hW} hold, then for all $\theta\in\Theta_0$ such that $\|\theta - \theta^\ast \| \geq \epsilon + \rho\sqrt{N_u/\beta}$,
for all $t\in\NN_{\geq 0}$ and any $\epsilon >0$, we have
\[
\Pr \{ \theta \in \Theta_t \}  \leq \biggl[ 1 - \probw \bigl(\epsilon \sqrt{\beta/N_u} \bigr) \biggr]^{\lfloor t/N_u \rfloor}  .
\]
\end{theorem}

\begin{corollary}\label{cor:min_set_hat}
Under Assumptions~\ref{ass:compact_disturbance}, \ref{ass:pe} and \ref{ass:hW}, 
$\Theta_t = \bigcap_{j=1}^t\Delta_j\cap\Theta_0$ converges to 
$\Theta_\infty \subseteq \{\theta^\ast\} \oplus \rho\sqrt{N_u/\beta}\,\B$ with probability~1.
\end{corollary}

\begin{theorem}\label{thm:fixed_set_what}
Let Assumptions~\ref{ass:compact_disturbance}, \ref{ass:pe} and~\ref{ass:hW} hold and let $\Theta_t$ be the fixed complexity parameter set defined by (\ref{eq:defn_Theta_set}), (\ref{eq:fixed_set_update}) with Remark~\ref{rem:periodic_update}. Then, for all $\theta\in\Theta_0$ such that  $[M_\Theta]_i (\theta- \theta^\ast) \geq \epsilon + \rho N_u \tau/\beta$ for some $i\in\NN_{[1,r]}$ and any $\epsilon > 0$, we have, for all $t\in\NN_{\geq 0}$,
\[
\Pr \{\theta \in\Theta_t \}
\leq 
\biggl\{ 1 - \biggl[ p_{w} \Bigl( \frac{\epsilon\beta}{ N_u\tau  } \Bigr) \biggr]^{N_u}\biggr\}^{\lfloor t/N_u \rfloor} .
\]
\end{theorem}

\begin{proof}
A bound on the probability that $\theta\in\Theta_{t-N_u}$ satisfying $[M_\Theta]_i (\theta- \theta^\ast) \geq \epsilon + \rho N_u \tau/\beta$ lies in $\Delta_{t-N_u+1}\cap\cdots\cap\Delta_t$ can be found using the same argument as the proof of Theorem~\ref{thm:fixed_set}. Thus choose $\hat{w}_j^0\in\partial\W$, $j\in\NN_{[t-N_u,t-1]}$ so that $g_j \in\N_{\W}(\hat{w}_j^0)$, where
\[
g_j^\top = -[M_\Theta]_i \biggl(\sum_{k=t-N_u}^{t-1} D_k^\top D_k \biggr)^{-1} D_j^\top ,
\]
and pick $w_j^0 \in\partial\Wtight$ so that $\|\hat{w}_j^0 - w_j^0\| \leq \rho$ for each $j\in\NN_{[t-N_u,t-1]}$. Then, from $[M_\Theta]_i(\theta - \theta^\ast) \geq \epsilon + \rho N_u \tau/\beta$ and Assumptions~\ref{ass:pe} and~\ref{ass:hW} we have
\begin{align*}
\sum_{j=t-N_u}^{t-1} g_j^\top\bigl[ D_j(\theta^\ast - \theta) + w_j - \hat{w}_j^0\bigr] 
&= [M_\Theta]_i(\theta - \theta^\ast) + \sum_{j=t-N_u}^{t-1} g_j^\top (w_j - w_j^0) + \sum_{j=t-N_u}^{t-1} g_j^\top (w_j^0 - \hat{w}_j^0) \\
& \geq \epsilon + \rho \frac{N_u \tau}{\beta} - \frac{N_u\tau}{\beta} \max_{j\in\NN_{[t-N_u,t-1]}} \| w_j - w_j^0\| 
- \rho \frac{N_u\tau}{\beta} .
\end{align*}
Therefore, if $\|w_j - w_j^0\| < \epsilon\beta/(N_u\tau)$ for all $j\in\NN_{[t-N_u,t-1]}$, then $\sum_{j=t-N_u}^{t-1} g_j^\top \bigl[ D_j (\theta^\ast - \theta) + w_j - \hat{w}_j^0 \bigr] > 0$ which implies $\theta\notin \bigcap_{j=t-N_u+1}^t \Delta_j$.
From Assumption~\ref{ass:hW} and the independence of the sequence $\{w_0,w_1,\ldots\}$ we therefore conclude that
\[
\Pr \{\theta\in\Theta_t\} \leq \biggl\{ 1 - 
\biggl[ \probw \Bigl( \frac{\epsilon\beta}{N_u\tau} \Bigr) \biggr]^{N_u}\biggr\} 
\Pr \{\theta\in\Theta_{t-N_u}\},
\]
and the result follows by applying this inequality $\lfloor t/N_u\rfloor$ times.
\end{proof}

\begin{corollary}\label{cor:fixed_set_hat}
Under Assumptions~\ref{ass:compact_disturbance}, \ref{ass:pe} and \ref{ass:hW},
the fixed complexity parameter set defined by (\ref{eq:defn_Theta_set}), (\ref{eq:fixed_set_update}) and Remark~\ref{rem:periodic_update} converges with probability~1 to a subset of $\{\theta: M_\Theta (\theta - \theta^\ast) \leq \rho N_u \tau/\beta\} $.
\end{corollary}

\subsection{System with measurement noise} 
Consider the system model with an unknown parameter vector $\theta^\ast$ and measurement noise $s_t$:
\begin{subequations}\label{eq:noisy_measurements}
\begin{align}
x_{t+1} = &\ A(\theta^\ast)x_t+B(\theta^\ast)u_t + w_t \\
y_{t} = &\ x_t+s_t 
\end{align}
\end{subequations}
where $y_t \in \RR^{n_x}$ is a measurement (or state estimate) and the noise sequence $\{s_0,s_1,\ldots\}$ has independent elements satisfying $s_t  \in \S$ for all $t \in \NN_{\geq 0}$.

\begin{assumption}[Measurement noise bounds]\label{ass:setS}
$\S$ is a  compact convex polytope with vertex representation $\S = \Co\{s^{(1)},\ldots,s^{(h)}\}$.
\end{assumption}


Due to the measurement noise, the unfalsified parameter set must be constructed at each time $t\in\NN_{\geq 0}$ using the available measurements $y_t$, $y_{t-1}$, the known control input $u_{t-1}$, and sets $\W$ and $\S$ bounding the disturbance and the measurement noise.
To be consistent with (\ref{eq:noisy_measurements}), $\theta^\ast$ must clearly lie in the set $\{\theta : y_t - D(y_{t-1}-s_{t-1},u_{t-1})\theta - d(y_{t-1}-s_{t-1},u_{t-1}) \in \W \oplus \S\}$, and the smallest unfalsified parameter set based on this information is given by 
\begin{subequations}\label{eq:mn_unfalsified_set}
\begin{gather}
\Delta_t = \Co\{\Delta_t^{(1)},\ldots,\Delta_t^{(h)}\} ,
\\
\Delta_t^{(j)} = \{\theta : y_{t}- D(y_{t-1} - s^{(j)},u_{t-1}) \theta - d(y_{t-1}-s^{(j)},u_{t-1})  \in \W\oplus \S\}\ \ \forall j \in \NN_{[1,h]}.
\end{gather}
\end{subequations}
Thus Assumption~\ref{ass:setS} implies that the unfalsified set $\Delta_t$ is a convex polytope and the parameter set $\Theta_t$ can be estimated using, for example, the update law (\ref{eq:defn_Theta_set}), (\ref{eq:fixed_set_update}) if $\S$ is known.

\begin{assumption}[Tight measurement noise and disturbance bounds]\label{ass:probw_s}
For all $w^0\in\partial\W$, $s^0\in\partial\S$ and $\epsilon > 0$ we have
\[
\Pr\Biggl\{ \biggl\| \begin{bmatrix} w_t - w^0 \\ s_t - s^0 \end{bmatrix} \biggr\| < \epsilon \Biggr\} \geq  p_{w,s} (\epsilon )
\]
where $p_{w,s}(\epsilon) > 0$ whenever $\epsilon > 0$.
\end{assumption}

Given Assumptions \ref{ass:setS} and \ref{ass:probw_s}, the results of Sections \ref{sec:min_set} and \ref{sec:fixed_set} apply with minor modifications. 
Define $\xi_t = w_t + s_t$, then Assumption~\ref{ass:probw_s} implies 
\[
\Pr \bigl\{ \|\xi_t - \xi_t^0\| < \epsilon \bigr\} \geq p_{w,s}\bigl(\epsilon/\sqrt{2}\bigr)
\]
for any given $\xi_t^0 = w_t^0 + s_t^0$ with $w_t^0\in\partial \W$ and $s_t^0 \in \partial \S$. 
This implies the following straightforward extensions of Theorems~\ref{thm:min_set} and~\ref{thm:fixed_set} and Corollaries~\ref{cor:min_set} and~\ref{cor:fixed_set}.

\begin{corollary}\label{cor:min_set_v}
Let Assumptions~\ref{ass:compact_disturbance}, \ref{ass:pe}, \ref{ass:setS} and~\ref{ass:probw_s} hold and $\Theta_t = \bigcap_{j = 1}^t \Delta_j \cap \Theta_0$, with $\Delta_j$ given by (\ref{eq:mn_unfalsified_set}). Then for all $\theta\in\Theta_0$ such that $\|\theta - \theta^\ast \| \geq \epsilon$, for all $t\in\NN_{\geq 0}$ and all $\epsilon > 0$, we have
\[
\Pr \{ \theta \in \Theta_t \}  \leq \biggl[ 1 - p_{w,s} \Bigl( \epsilon \sqrt{\frac{\beta}{2N_u}} \Bigr) \biggr]^{\lfloor t/N_u \rfloor} .
\]
\end{corollary}

\begin{corollary}\label{cor:fixed_set_v}
Let Assumptions~\ref{ass:compact_disturbance}, \ref{ass:pe}, \ref{ass:setS} and~\ref{ass:probw_s} hold and let $\Theta_t$ be the fixed complexity parameter set defined by (\ref{eq:defn_Theta_set}), (\ref{eq:fixed_set_update}) with Remark~\ref{rem:periodic_update} and (\ref{eq:mn_unfalsified_set}). Then for all $\theta\in\Theta_0$ such that  $[M_\Theta]_i (\theta- \theta^\ast) \geq \epsilon$ for some $i\in\NN_{[1,r]}$ and any $\epsilon > 0$, we have, for all $t\in\NN_{\geq 0}$,
\[
\Pr \{ \theta\in\Theta_t \} \leq 
\biggl\{ 1 - \biggl[ p_{w,s} \Bigl( \frac{\epsilon\beta}{ \sqrt{2} N_u \tau} \Bigr) \biggr]^{N_u}\biggr\}^{\lfloor t/{N_u} \rfloor} .
\]
\end{corollary}

\begin{corollary}\label{cor:fixed_set_mn}
Under Assumptions~\ref{ass:compact_disturbance}, \ref{ass:pe}, \ref{ass:setS}, \ref{ass:probw_s} the parameter set $\Theta_t$ defined in Corollary~\ref{cor:min_set_v} or \ref{cor:fixed_set_v} converges to $\{\theta^\ast\}$ with probability 1.
\end{corollary}

 {
\begin{remark}
If measurement noise is present, modifications to the proposed algorithm in section \ref{subsec:proposed_algorithm} are needed. If no additional output constraints are present, then, given the noisy measurement $y_t$, 
the constraint (\ref{eq:initial_con}) should be replaced by
\[
T \blue{x} \leq \alpha_0, \ \forall x \in \{y_t\}\oplus (-\S)
\]
in order to ensure robust satisfaction of input and state constraints. 
In addition, the unfalsified parameter set $\Delta_t$ is in this case given by (\ref{eq:mn_unfalsified_set}). 
\end{remark}}

\section{Numerical examples} \label{sec:example}
This section presents simulations to illustrate the operation of the proposed  adaptive robust MPC scheme. The section consists of two parts. The first part investigates the effect of additional weight $\gamma$ in optimization problem $\mathcal{P}$ by using the example of a second-order system from~\cite{Lorenzen2017}. The second part demonstrates the relationship between the speed of parameter convergence and minimal eigenvalue $\beta $ from the PE condition. 

\subsection{Objective function with weighted PE condition}
Consider the second-order discrete-time uncertain linear system from~\cite{Lorenzen2017}, with model parameters
\begin{gather*}
A_0 = \begin{bmatrix} 0.5 & 0.2 \\ -0.1 & 0.6 \end{bmatrix}, \ 
A_1 = \begin{bmatrix} 0.042 & 0 \\ 0.072 & 0.03 \end{bmatrix}, \
A_2 = \begin{bmatrix} 0.015 & 0.019 \\ 0.009 & 0.035 \end{bmatrix},  \ 
A_3 = \begin{bmatrix} 0 & 0 \\ 0 & 0 \end{bmatrix}, \\
B_0 = \begin{bmatrix} 0 \\ 0.5 \end{bmatrix}, \ 
B_1 = \begin{bmatrix} 0 \\ 0 \end{bmatrix}, \ 
B_2  = \begin{bmatrix} 0 \\ 0 \end{bmatrix}, \ 
B_3 = \begin{bmatrix} 0.0397 \\ 0.059 \end{bmatrix} ,
\end{gather*}
and with true system parameter $\theta^* = \begin{bmatrix} 0.8 & 0.2 &  {-0.5}\end{bmatrix}^\top$. The initial parameter set estimate is $\Theta_0 = \{ \theta : \| \theta \|_\infty \leq 1 \}$,  {and for all $t\geq 0$, $\Theta_t$ is a \blue{hyperrectangle}, with $\M_\theta = [\Id \ \ {-\Id}]^\top$ }.
The elements of the disturbance sequence $\{w_0,w_1,\ldots\}$ are independent and identically (uniformly) distributed on $\mathcal{W} = \{ w  \in \mathbb{R}^2 : \ \| w\|_\infty \leq 0.05\}$. The state and input constraints are $[x_t]_2 \geq -0.3$ and $u_t\leq 1$.  {The MPC prediction horizon and PE window length are set to be $N = 10$ and $N_u = 2$ respectively. The matrix $T$ is chosen according to Remark \ref{rem:T} and has 9 rows. }

All simulations were performed in Matlab on a 3.4\,GHz Intel Core i7 processor, and the online MPC optimization $\mathcal{P}$ was solved using Mosek \cite{Mosek}. For purposes of comparison, the same parameter set update law and nominal parameter update law were used in all cases. Robust satisfaction of input and state constraints and recursive feasibility were observed in all simulations, in agreement with Proposition~\ref{prop:recursive}.
\blue{To illustrate satisfaction of the state constraint $[x]_2 \geq -0.3$, Figure \ref{fig:initial_conditions} shows the cross-sections of the robust state tube predicted at $t = 0$, with initial condition $x_0 = (3,6)$, together with the closed-loop state trajectories for 10 different initial conditions. }
 
 \begin{figure}[htb]
\setlength{\fboxsep}{0pt}%
\setlength{\fboxrule}{0pt}%
\centerline{\includegraphics[scale=0.7, trim = 0mm 10mm 0mm 14mm, clip=true]{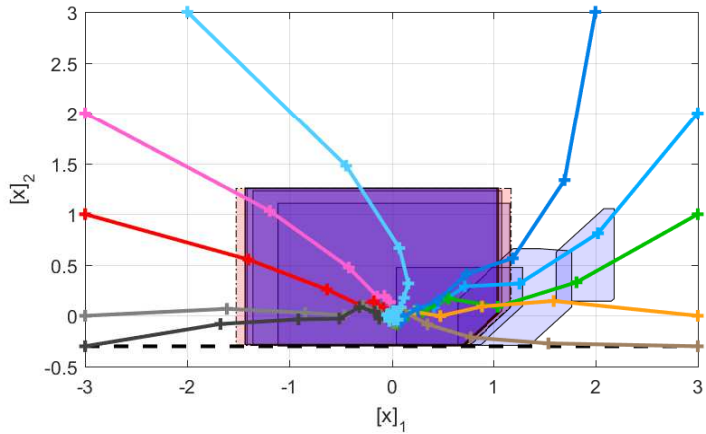}}
\caption{Closed-loop trajectories (solid lines) from different initial conditions, and predicted state tube cross-sections $\{\X_1, \dots, \X_{N}\}$ at $t=0$ for initial condition $x_0 = (3, 2)$, with $\X_N$ shown in red and enclosed by dashed line.}
\label{fig:initial_conditions}
\end{figure}
\vspace{-8mm}

\renewcommand{\arraystretch}{1.3}
\begin{table}[h!] 
\caption{Performance comparison of robust MPC algorithms, with and without PE conditions}
 {
\begin{tabular}{m{3.5cm}| m{2.6cm} m{2.6cm}  m{3.1cm}  m{3.1cm}  }
\hline
 & (A) & (B)& (C)& (D)\\
Algorithm &\blue{Homothetic tube (no PE)\cite{Lorenzen2018}*} &  Homothetic tube with PE \cite{Lorenzen2018} &  \blue{Proposed Algorithm (no PE)*}
  & Proposed Algorithm with PE ($\gamma = 10^3 $) \\
\hline
Yalmip time/s & \blue{0.1825} &0.7407 &\blue{0.2053} &0.2608 \\
Solver time/s &\blue{0.1354} &0.2065 &\blue{0.1016} &0.0766\\  
Computational time/s (Yalmip + solver time)&\blue{0.3179} &0.9472 &\blue{0.3069} & 0.3374\\
$\Theta_{100}$ set size /\%  &18.26 &{18.51} & 18.26 & 16.56 \\
\hline
\end{tabular}}
\label{table:compare}\\
\blue{*For algorithms without PE constraint, a QP solver, Gurobi~\cite{Gurobi}, is used instead of Mosek.} 
\end{table}

 {Table~\ref{table:compare} compares the the computational time and parameter sizes of the proposed algorithm and existing algorithms when the same initial conditions and disturbance sequences $\{w_0, w_1, \dots \}$ are used. Algorithm (A) refers to the robust adaptive MPC in Section 3.4 of Lorenzen et al.~\cite{Lorenzen2018}. Algorithm (B) is a modification of algorithm (A) that incorporates the PE constraint  $\sum^{N_u}_{l=0} \textbf{u}_{t-l}\textbf{u}_{t-l}^\top \geq \beta \Id$, which is implemented as described in Lu and Cannon~\cite{Lu2019} with a fixed $\beta$ value: $\beta=10^{-4}$. Algorithms (C) and (D) are the algorithm proposed in Section \ref{subsec:proposed_algorithm}, with and without the PE condition, respectively.  }

 Consider first algorithms (A) and (C) in Table~\ref{table:compare}. 
\blue{Although (C) uses a more flexible tube representation, there is negligible difference in overall computational time relative to (A).  This is due to the use of a more efficient method of enforcing constraints on predicted tubes in (A) than (C), which introduces additional optimization variables to enforce these constraints.
The more flexible tube representation employed in (C) 
provides a larger terminal set, as shown in Figure~\ref{fig:term_sets}.}
%
%
Moreover, the formulation of (C) incorporates information on $\Theta_t$ in the constraints, and as a result, the terminal set increases in size over time as the parameter set $\Theta_t$ shrinks. 
On the other hand, the homothetic tube MPC employed in (A) employs a terminal set that is computed offline and is not updated online. 

 {Comparing algorithms (B) and (D) in Table~\ref{table:compare}, it can be seen that implementing the PE condition using an augmented cost function and linearized constraint (as in (D)) results in lower computation and faster parameter convergence than using a PE constraint with a fixed value of $\beta$ (as in (B)). 
Tuning the value of $\beta$ in algorithm (B) is challenging, since a value that is too small results in slow convergence whereas choosing $\beta$ too large frequently causes the PE constraint to be infeasible. Moreover, whenever the PE constraint is infeasible in algorithm (B), the MPC optimization is solved a second time without the PE constraint, thus increasing computation.}
\begin{figure}[htb]
\centerline{\includegraphics[scale=0.5, trim = 6mm 2mm 12mm 4mm, clip=true]{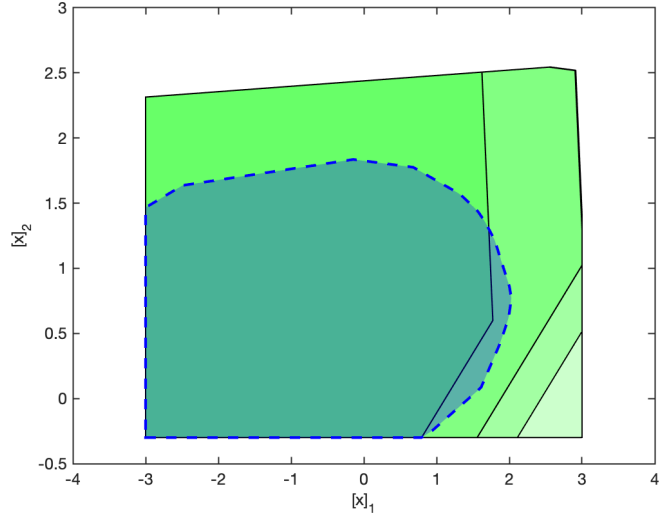}}
\caption{ {Terminal sets for Algorithm (C) at times $t = 0, 1, 100, 1000$ and the terminal set for (A) (which is computed offline and not updated online).
The terminal sets for (C) are shown in green with solid boundaries, and are nested and increasing over time.
The terminal set for (A) is shown in blue with dashed line boundary. }}
\label{fig:term_sets}
\end{figure}


\begin{figure}[htb]
\setlength{\fboxsep}{0pt}%
\setlength{\fboxrule}{0pt}%
\centerline{\includegraphics[scale=0.7, trim = 6mm 0mm 8mm 4mm, clip=true]{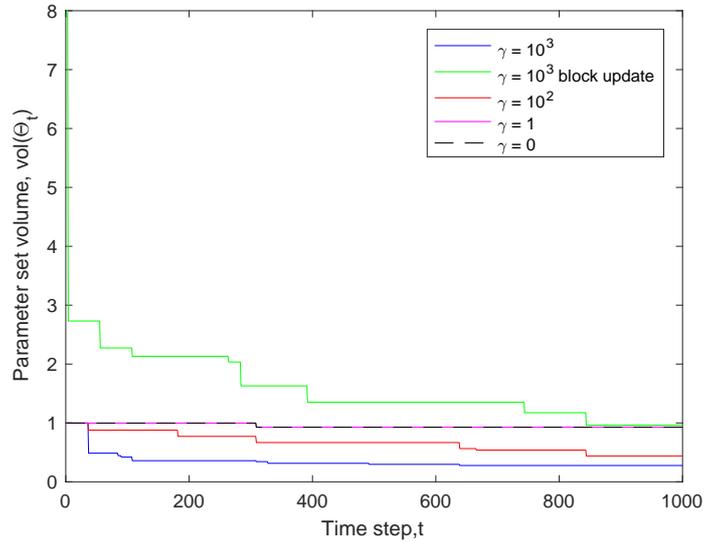}}
\caption{Volume of parameter set $\Theta_t$ over time for a range of weights $\gamma$ in the MPC objective function~(\ref{eq:cost_J2}).}
\label{fig:closed_loop_different_gamma}
\end{figure}

Figure \ref{fig:closed_loop_different_gamma} shows the effect of the weighting coefficient $\gamma$ in the objective function~(\ref{eq:cost_J2}) on the parameter set $\Theta_t$ when the same initial conditions  {($x_0 =  [3,4]^\top$, $\Theta_0$, $\bar{\theta}_0$)} and disturbance sequences $\{w_0,w_1,\ldots\}$ are used. Larger values of $\gamma$ place greater weighting on $\beta$ in the MPC cost (\ref{eq:cost_J2}), and thus on satisfaction of the PE condition (\ref{eq:PE_condition}). Therefore increasing $\gamma$ results in a faster convergence rate in the parameter set volume.  {When the same weighting coefficient $\gamma$ is used, performing the parameter set update periodically (as discussed in Remark \ref{rem:periodic_update}) slows down the convergence rate of the parameter set, as shown by the green line. 
For this simulation, the parameter set update (online step 2 in Section \ref{subsec:proposed_algorithm}) takes only 2\% of the total computational time. }


The relationship between weighting coefficient $\gamma$ and volume of parameter set $\Theta_t$ is illustrated in Figure \ref{fig:volume_against_gamma}.  
For values of $\gamma$ between $10^{-3}$ and $10^3$, closed loop simulations were performed with the same initial conditions, disturbance sequences, and initial nominal model and parameter set. The parameter set volume after 20 time-steps is shown. Figure \ref{fig:volume_against_gamma} also shows that increasing $\gamma$ results in a faster parameter set convergence rate, in agreement with Figure~\ref{fig:closed_loop_different_gamma}. 
%

\begin{figure}[htb]
\setlength{\fboxsep}{0pt}%
\setlength{\fboxrule}{0pt}%
\centerline{\includegraphics[scale=0.5, trim = 6mm 0mm 8mm 4mm, clip=true]{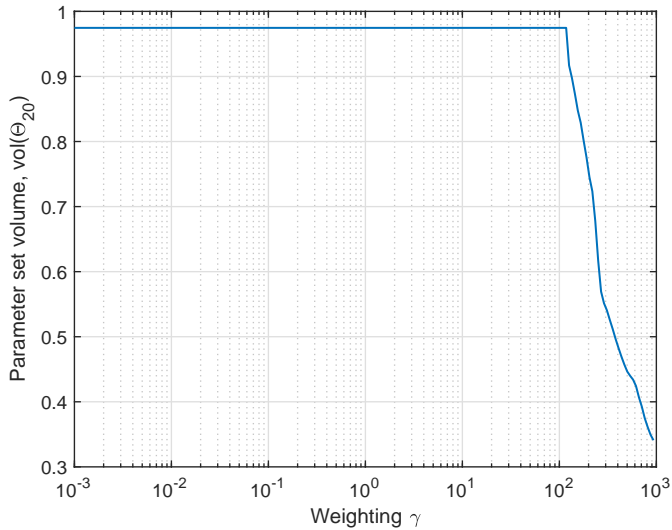}}
\caption{Volume of parameter set $\Theta_{20}$ against weighting $\gamma$ in the MPC objective function (\ref{eq:cost_J2}).}
\label{fig:volume_against_gamma}
\end{figure}

For the same set of simulations, Figure \ref{fig:beta_against_gamma} shows the optimal value of $\beta$ in (\ref{eq:PE_constraint}) and (\ref{eq:PE_condition}) against $\gamma$. From (\ref{eq:cost_J2}), it is expected that a larger $\gamma$ value will increase the influence of the term $-\gamma \beta$, thus pushing $\beta$ to be more positive. The left-hand figure shows the value of $\beta$ in the convexified constraint (\ref{eq:PE_constraint}). As expected, the increase in $\gamma$ leads to a smooth increase in $\beta$ initially, but after a certain point, any further increase in the weighting factor $\gamma$ does not affect the calculated $\beta$ value. The right-hand figure shows the value of $\beta_1$ in the PE condition (\ref{eq:PE_condition}). 
 {The difference between $\beta$ and $\beta_1$ illustrates the conservativeness of the convexification proposed in Section \ref{sec:pe_cost}. 
Note that this can be reduced by repeating steps (3) and (4) in the online part of the proposed algorithm, thus iteratively re-computing the reference sequences $\hat{\bf x}$, $\hat{\bf u}$ and reducing the conservativeness of linearisation to any desired level.}
It is interesting to note that, although the optimal value of $\beta$ in (\ref{eq:PE_constraint}) levels off at $\gamma = 1$, the value of $\beta_1$ in the PE condition (\ref{eq:PE_condition}) increases monotonically between $\gamma = 10$ and $\gamma = 10^3$. 
The smaller $\beta$ values observed with (\ref{eq:PE_condition}) also explain the lower rates of parameter convergence for small values of $\gamma$ in Figure \ref{fig:volume_against_gamma}. In practice, it can be used as a guideline for the tuning of $\gamma$.

\begin{figure}[!ht]
\setlength{\fboxsep}{0pt}%
\setlength{\fboxrule}{0pt}%
\centerline{%
\includegraphics[scale=0.5,trim = 4mm 0mm 14mm 7mm, clip]{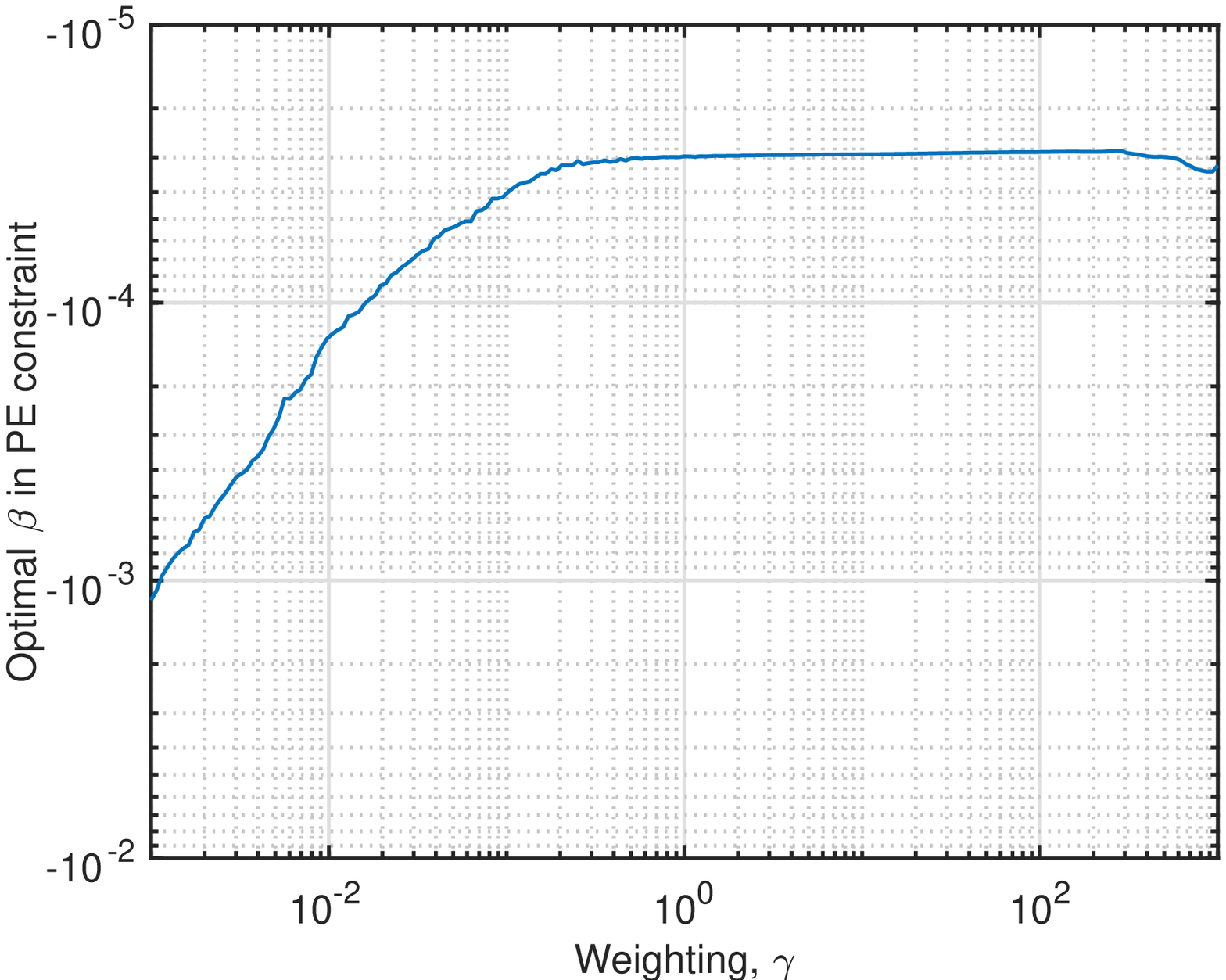}%
\includegraphics[scale=0.5,trim = 5mm 0mm 16mm 7mm, clip]{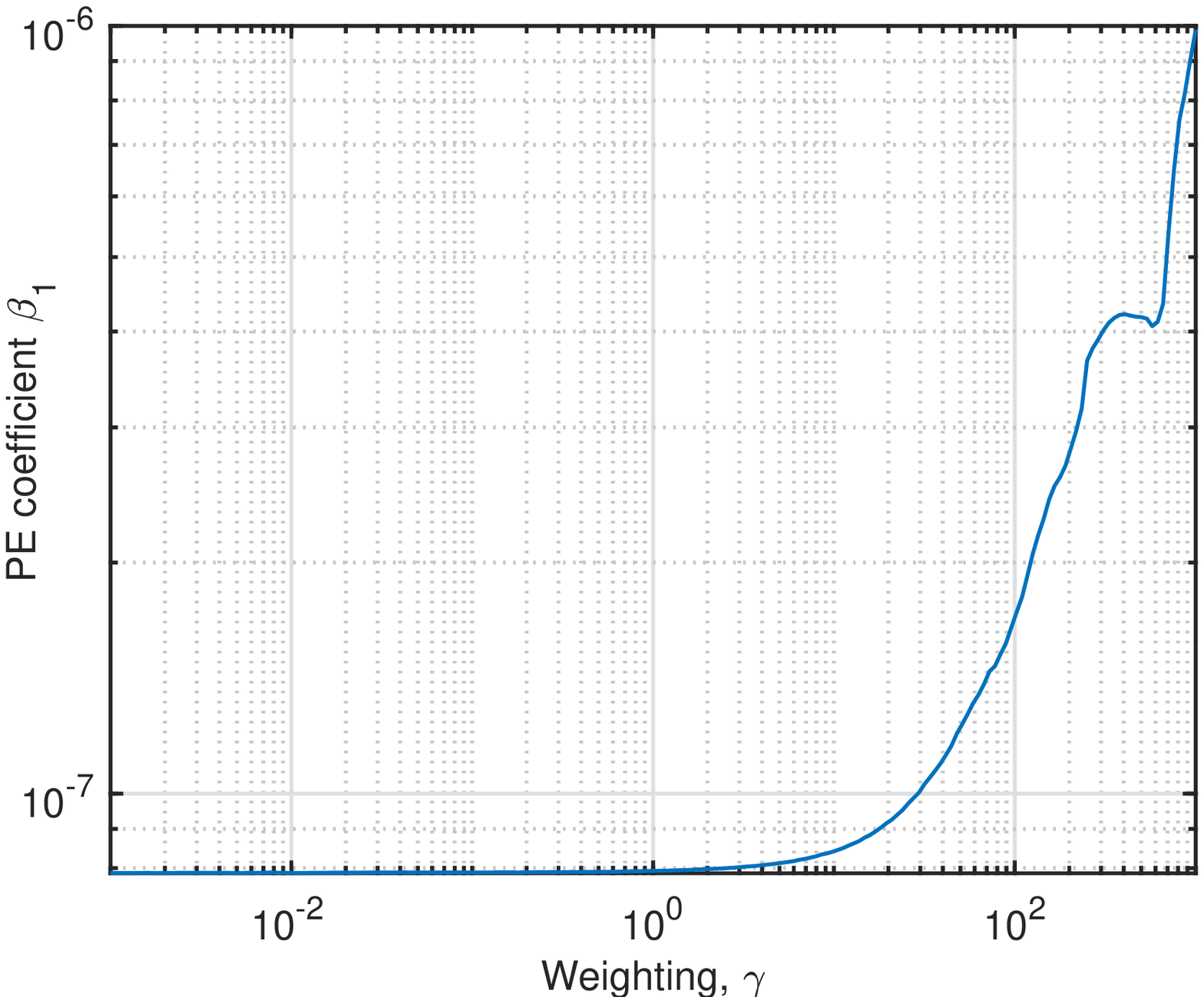}}
\caption{Degree of satisfaction of conditions for persistency of excitation as a function of the weighting $\gamma$ in the MPC objective function (\ref{eq:cost_J2}). Left: optimal value of $\beta$ in the constraint (\ref{eq:PE_constraint}). Right: computed value of the PE coefficient $\beta_1$ in (\ref{eq:PE_condition}).}
\label{fig:beta_against_gamma}
\end{figure}

 {Table \ref{table:volume} illustrates the convergence of the estimated parameter set over a large number of time steps for the initial condition $x_0 = [2,3]^\top$ and a randomly generated disturbance sequence $\{w_0, w_1, \dots\}$. Here $\gamma$ was chosen as $10^3$ to speed up the convergence process. In agreement with Theorem \ref{thm:fixed_set}, $\Theta_t$ has shrunk to a small region around the true parameter value at $t = 5000$.}

\vspace{-3mm}
\renewcommand{\arraystretch}{1.3}
\begin{table}[h!] 
\caption{Asymptotic convergence of the estimated parameter set}
 {
\centerline{\begin{tabular}{ c| c c c c c c c c }
\hline
Time Step /t & 0 & 1 & 5 & 50 & 100 & 500 & 1000 & 5000\\
\hline
$\Theta_t$ set size /\% & 100 & 30.21 &18.50 & 14.46 & 12.75 & 11.11 & 1.51 & 0.25\\
\hline
 \end{tabular}}}
\label{table:volume}
\end{table}

\subsection{Relationship between PE coefficient and convergence rate}
We next consider third-order discrete-time linear systems given by (\ref{eq:update_equation})
with $x \in \RR^3$, $u \in \RR^2$, $\theta \in \RR^3$ and 
\[
\W = \{w: \|w \|_\infty \leq 0.1\}.
\]
The system matrices $\big(A(\theta), B(\theta)\big)$ satisfy (\ref{eq:system_model}) with randomly generated $A_i$, $ B_i$, $\theta^\ast$ parameters and initial parameter set 
$\Theta_0 = \{\theta: \| \theta\|_\infty \leq 0.25\}$. 
%
In each case the estimated parameter sets $\Theta_t$ have fixed complexity, with face normals aligned with the coordinate axes in parameter space. 
A linear feedback law is applied, $u_t = Kx_t$, where $K$ is a stabilizing gain. We use these systems to investigate the relationship between the coefficient $\beta_1$ in the PE condition~(\ref{eq:PE_condition}) and rate of convergence of the estimated parameter set.

Taking the window length in~(\ref{eq:PE_condition}) to be $N_u = 10$, closed-loop trajectories were computed for 10 time steps
and the parameter set $\Theta_t$ was updated according to (\ref{eq:fixed_set_update}).
Simulations were performed for 500 different initial conditions, and the average value of $\beta_1$ was computed for each initial condition using 100 random disturbance sequences $\{w_0,w_1,\ldots\}$. 
Figure \ref{fig:average_against_beta} illustrates the relationship between the average size of the identified parameter set $\Theta_t$ and the average value of $\beta_1$ in the PE condition~(\ref{eq:PE_condition}). Clearly, increasing $\beta_1$ results in a smaller parameter set on average, and hence a faster rate of convergence of $\Theta_t$, which is consistent with the analysis of Section~\ref{sec:fixed_set}. The inner and outer radii shown in the figure on the left are the radii of the smallest and largest spheres, respectively, that contain and are contained within the parameter set estimate after 10 time steps. 
A similar trend can also be seen between the average volume of the parameter set $\Theta$ and the ensemble average value of $\beta_1$. 

\begin{figure}[!ht]
\setlength{\fboxsep}{0pt}%
\setlength{\fboxrule}{0pt}%
\centerline{%
\includegraphics[scale=0.5,trim = 4mm 0 14mm 5mm, clip]{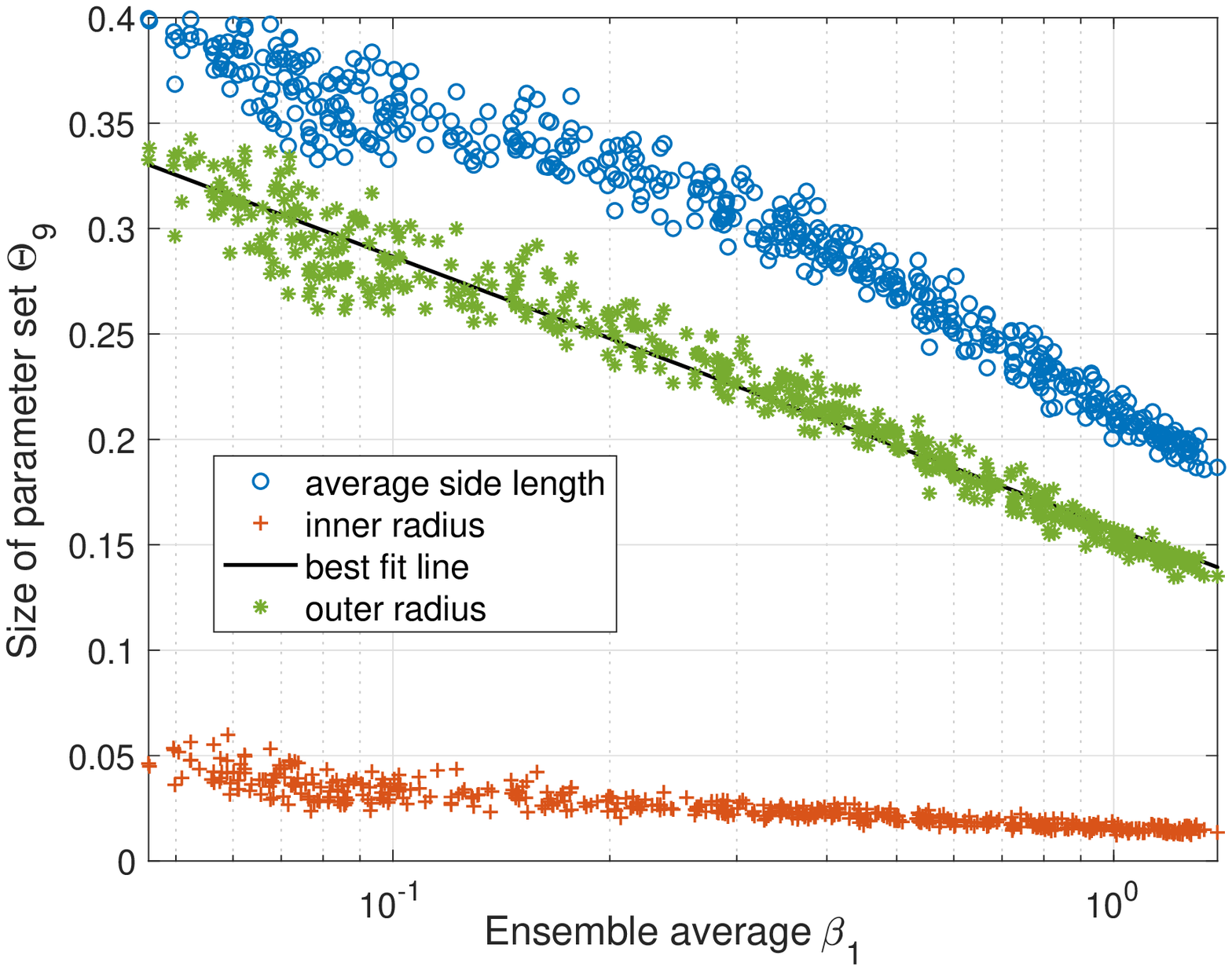}%
\includegraphics[scale=0.5, trim = 0 0 16mm 5mm, clip]{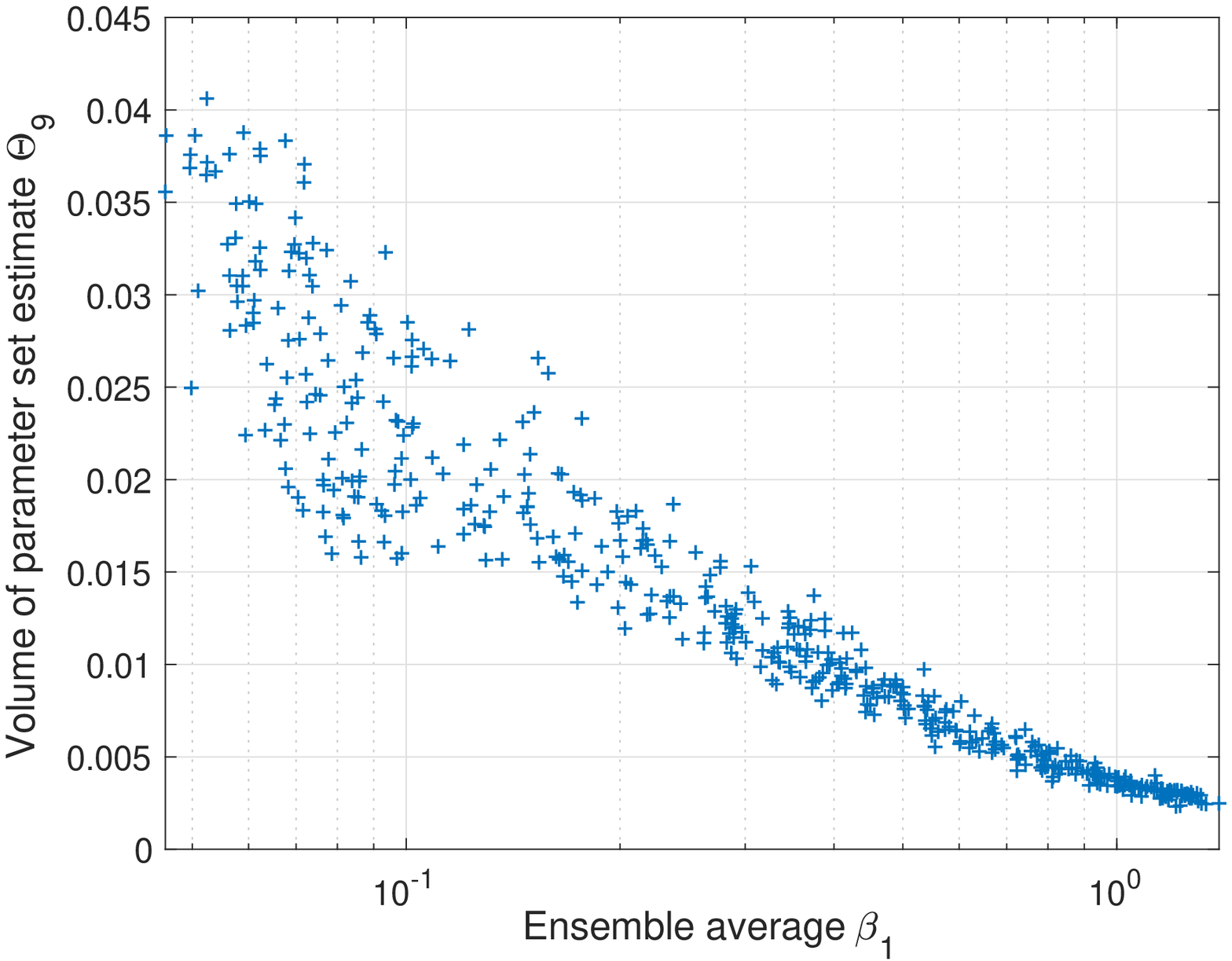}%
}
\caption{Average size of parameter set after 10 time steps against average value of $\beta_1$ in the PE condition (\ref{eq:PE_condition}) with $N_u=10$. The parameter set size and $\beta_1$ were computed for 500 different initial conditions. Left: mean side length and inner and outer radii of $\Theta_9$. Right: volume of $\Theta_9$.}
\label{fig:average_against_beta}
\end{figure}



\section{Conclusions}
In this paper we propose an adaptive robust MPC algorithm that combines robust tube MPC and set membership identification. 
The MPC formulation employs a nominal performance index and guarantees robust constraint satisfaction, recursive feasibility and input-to-state stability. A convexified persistent excitation condition is included in the MPC objective via a weighting coefficient, and the relationship between this weight and the convergence rate of the estimated parameter set is investigated. For computational tractability, a fixed complexity polytope is used to approximate the estimated parameter set. The paper proves that the parameter set will converge to the vector of system parameters with probability 1 despite this approximation. Conditions for convergence of the estimated parameter set are derived for the case of inexact disturbance bounds and noisy measurements. 
Future work will consider systems with stochastic model parameters and probabilistic constraints. 
Quantitative relationships between the convergence rate of the estimated parameter set and conditions for persistency of excitation will be investigated further and methods of enforcing persistency of excitation of the closed loop system will be considered.


\bibliographystyle{plain}
\bibliography{Sample1_edited}

\end{document}